\documentclass[a4paper, 12pt]{article}

\usepackage[utf8]{inputenc}
\usepackage[T1]{fontenc}
\usepackage[a4paper, margin=2.5cm]{geometry}
\usepackage{needspace}
\usepackage{libertine} 
\usepackage{inconsolata} 
\usepackage{amsmath, amsthm, amssymb}
\usepackage{graphicx}
\usepackage{enumerate}
\usepackage{authblk}
\usepackage[textsize=small, textwidth=2cm, color=yellow]{todonotes}
\usepackage{thmtools}
\usepackage{thm-restate}
\usepackage[colorlinks=true, linkcolor =blue!80, citecolor=orange!70!black]{hyperref}
\usepackage{float}
\usepackage[capitalize]{cleveref}
\usepackage{comment}
\usepackage{verbatim}
\usepackage{xspace}
\usepackage{tikz}
\usepackage{soul}
\usepackage{subcaption}

\renewenvironment{abstract}
{\small\vspace{-1em}
\begin{center}
\bfseries\abstractname\vspace{-.5em}\vspace{0pt}
\end{center}
\list{}{
\setlength{\leftmargin}{0.6in}%
\setlength{\rightmargin}{\leftmargin}}%
\item\relax}
{\endlist}

\declaretheorem[name=Theorem, numberwithin=section]{theorem}
\declaretheorem[name=Lemma, sibling=theorem]{lemma}

\declaretheorem[name=Claim, sibling=theorem]{claim}

\declaretheorem[name=Remark, style=remark, sibling=theorem]{remark}

\def\cqedsymbol{\ifmmode$\lrcorner$\else{\unskip\nobreak\hfil
\penalty50\hskip1em\null\nobreak\hfil$\lrcorner$
\parfillskip=0pt\finalhyphendemerits=0\endgraf}\fi}

\Crefname{theorem}{Theorem}{Theorems}
\Crefname{lemma}{Lemma}{Lemmas}
\Crefname{claim}{Claim}{Claims}

\newcommand{\sg}[1]{\left\{#1\right\}}

\newcommand{\ceil}[1]{\left \lceil #1 \right \rceil}
\newcommand{\Gi}{G_i}
\newcommand{\Gell}{G_{\ell}}
\newcommand{\Gip}{G_{i+1}}
\newcommand{\Ri}{R_i}
\newcommand{\Rell}{R_{\ell}}
\newcommand{\Rip}{R_{i+1}}
\newcommand{\Gami}{\Gamma_i}
\newcommand{\Gamell}{\Gamma_{\ell}}
\newcommand{\Gamip}{\Gamma_{i+1}}
\newcommand{\mui}{\mu_i}
\newcommand{\muell}{\mu_{\ell}}
\newcommand{\muip}{\mu_{i+1}}
\newcommand{\Xip}{X_{i+1}}
\newcommand{\Xell}{X_{\ell}}
\newcommand{\fip}{f_{i+1}}
\newcommand{\np}{\mathrm{np}}

\newcommand{\ie}{\emph{i.e.},\xspace}
\let\leq\leqslant
\let\geq\geqslant

\title{Largest planar graphs of diameter $3$ and fixed maximum degree -- connection with fractional matchings}

\author[1]{Antoine Dailly}

\author[2]{Sasha Darmon}

\author[3]{Ugo Giocanti\thanks{Supported by the National Science Center of Poland
under grant 2022/47/B/ST6/02837 within the OPUS 24 program, and partially supported by the French ANR
Project TWIN-WIDTH (ANR-21-CE48-0014-01).}}

\author[4]{Claire Hilaire\thanks{Partially supported by Slovenian Research and Innovation Agency (research project J1-4008).}}
\author[5]{Petru Valicov}

\affil[1]{Université Clermont Auvergne, INRAE, UR TSCF, 63000, Clermont-Ferrand, France}
\affil[2]{LBBE, Université Lyon 1, CNRS, Lyon, France \& INRIA, Université Lyon 1, Lyon, France}
\affil[3]{Faculty of Mathematics and Computer Science, Jagiellonian University,
Kraków, Poland}
\affil[4]{FAMNIT, University of Primorska, Koper, Slovenia}
\affil[5]{LIRMM, Université de Montpellier, CNRS, Montpellier, France.}

\begin{document}

\maketitle

\begin{abstract}
 The degree diameter problem asks for the maximum possible number of vertices in a graph of maximum degree $\Delta$ and diameter $D$. In this paper, we focus on planar graphs of diameter $3$. Fellows, Hell and Seyffarth (1995) proved that for all $\Delta\geq 8$, the maximum number $\np_{\Delta, D}$ of vertices of a planar graph with maximum degree at most $\Delta$ and diameter at most 3 satisfies $\frac{9}{2}\Delta - 3 \leq \np_{\Delta,3} \leq 8 \Delta + 12$. We show that the lower bound they gave is optimal, up to an additive constant, by proving that there exists $c>0$ such that $\np_{\Delta,3} \leq \frac{9}{2}\Delta + c$ for every $\Delta\geq 0$. Our proof consists in a reduction to the fractional maximum matching problem on a specific class of planar graphs, for which we show that the optimal solution is $\tfrac{9}{2}$, and characterize all graphs attaining this bound.
\end{abstract}

\section{Introduction}

\paragraph*{The degree diameter problem.}
 Given two integers $\Delta$ and $D$, the degree diameter problem consists in finding the largest possible number of vertices $n_{\Delta, D}$ in a graph of maximum degree $\Delta$ and diameter $D$. This problem has been actively studied since the 1960's, and we refer to the exhaustive survey of Miller and \v{S}ir\'a\v{n}~\cite{MS13} for a detailed overview.

 An easy upper bound on $n_{\Delta,D}$, called the \emph{Moore bound} $M_{\Delta, D}$, can be established by observing that for every graph $G$ of diameter at most $D$ and maximum degree $\Delta$ and every vertex $v\in V(G)$, the number of vertices at distance at most $i\in \sg{1, \ldots, D-1}$ from $v$ in $G$ is at most $\Delta(\Delta-1)^{i-1}$, giving for every $\Delta\geq 2, D\geq 1$ the inequality: 
 
\begin{equation*}
     n_{\Delta, D} \leq M_{\Delta, D}:=
     \begin{cases} 
          1 + \Delta\frac{(\Delta-1)^{D}-1}{\Delta-2} & \text{if } \Delta > 2, \\
          2D+1 & \text{if } \Delta = 2. 
     \end{cases}
\end{equation*}

Graphs whose number of vertices is exactly $M_{\Delta, D}$ are called \emph{Moore graphs}, and a large amount of work has been dedicated to establish an exhaustive list of Moore graphs. In particular, it has been proven that only finitely many Moore graphs exist, and all of them most have diameter at most $3$ (moreover, the only Moore graph with diameter $3$ is $C_7$) see~\cite{HS60, Damerell, BI73}. Furthermore, the only possible Moore graph whose existence is still an open question should have diameter $2$ and maximum degree exactly $57$, see~\cite{DALFO2019survey} for a detailed survey on this question. 

Determining the asymptotics of $n_{\Delta, D}$ is also a widely open question, at the moment the best known lower bounds follow from a construction of Canale and G\'omez \cite{CG05}, who proved that for every large enough integer $D$, there are infinitely many values $\Delta$ for which there exists a graph with order at least 
$\left(\tfrac{\Delta}{1.6}\right)^{D}$. 

\paragraph*{The degree diameter problem in planar graphs.}
Over the last decades, variants of the degree diameter problem have been considered when one restricts to some specific classes of graphs. In this paper, we will focus on planar graphs, however it is worth mentioning that the degree diameter problem has also been studied in graphs of bounded genus, and more systematically in sparse classes \cite{SiSi,PVW15,NEVO2016}. 

For every $\Delta\geq 1, D\geq 1$, we let $\np_{\Delta, D}$ denote the maximum possible number of vertices in a planar graph with diameter at most $D$ and maximum degree $\Delta$. 
The study of the degree diameter problem in planar graphs was initiated by Fellows and Hell~\cite{HS93}, who proved that for all $\Delta \geq 8$, $\np_{\Delta, 2}=\lfloor \tfrac32\rfloor +1$. Using the planar separator theorem by
Lipton and Tarjan~\cite{LiptonTarjan}, 
Fellows, Hell and Seyffarth \cite{FHS95} proved that there exists a constant $c>0$ such that for all large enough $D, \Delta$, we have $\np_{\Delta, D}\leq cD\Delta^{\lfloor\frac{D}{2}\rfloor}$. Nevo, Pineda-Villavicencio and Wood \cite{NEVO2016} proved further that the factor $D$ can be removed in the previous inequality, in the sense that there exists $f:\mathbb N\to \mathbb N$ such that $\np_{\Delta, D}\leq c\Delta^{\lfloor\frac{D}{2}\rfloor}$ for all $D\geq 1$ and $\Delta\geq f(D)$.
When $D$ is even, Tishchenko \cite{TISHCHENKO2012} proved that for every $\Delta\geq 6(12D+1)$, we have $\np_{\Delta, D}=(\tfrac32+o(1))\Delta^{\tfrac{D}{2}}$. In fact, Tishchenko even computed an exact formula for $\np_{\Delta, D}$ when $D$ is even and $\Delta\geq f(D)$ for some linear function $f$. 

On the other hand, when $D$ is odd, the general upper bound of Nevo, Pineda-Villavicencio and Wood remains the best known. When $D=3$, Fellows, Hell and Seyffarth \cite{FHS95} explicitly asked for the value of $\np_{\Delta, 3}$, and proved the following.

\begin{theorem}[\cite{FHS95}]
\label{thm:previous_general_bound}
For every $\Delta\geq 4$, 
$$\left\lfloor\frac{9}{2}\Delta\right\rfloor-3 \leq \np_{\Delta,3} \leq 8\Delta + 12.$$
\end{theorem}

The construction of planar graphs attaining the lower bound of \Cref{thm:previous_general_bound} given in \cite{FHS95}, is depicted in the left of \Cref{fig: lower-bound}. Note that, at the moment, this lower bound remains the best known for graphs with $\Delta\geq 5$. For smaller values of $\Delta$, it is known that $\np_{3,3}=12$~\cite{P96,TISHCHENKO2001} and $\np_{4,3}\geq 16$~\cite{Wiki}. 
However, we note that constructions that generalize in a natural way those attaining $\np_{3,3}$ or $\np_{4,3}$ only produce planar graphs of diameter 3 having $4\Delta+O(1)$ vertices.

\begin{figure}[!ht]
    \centering
    \includegraphics[scale=0.7]{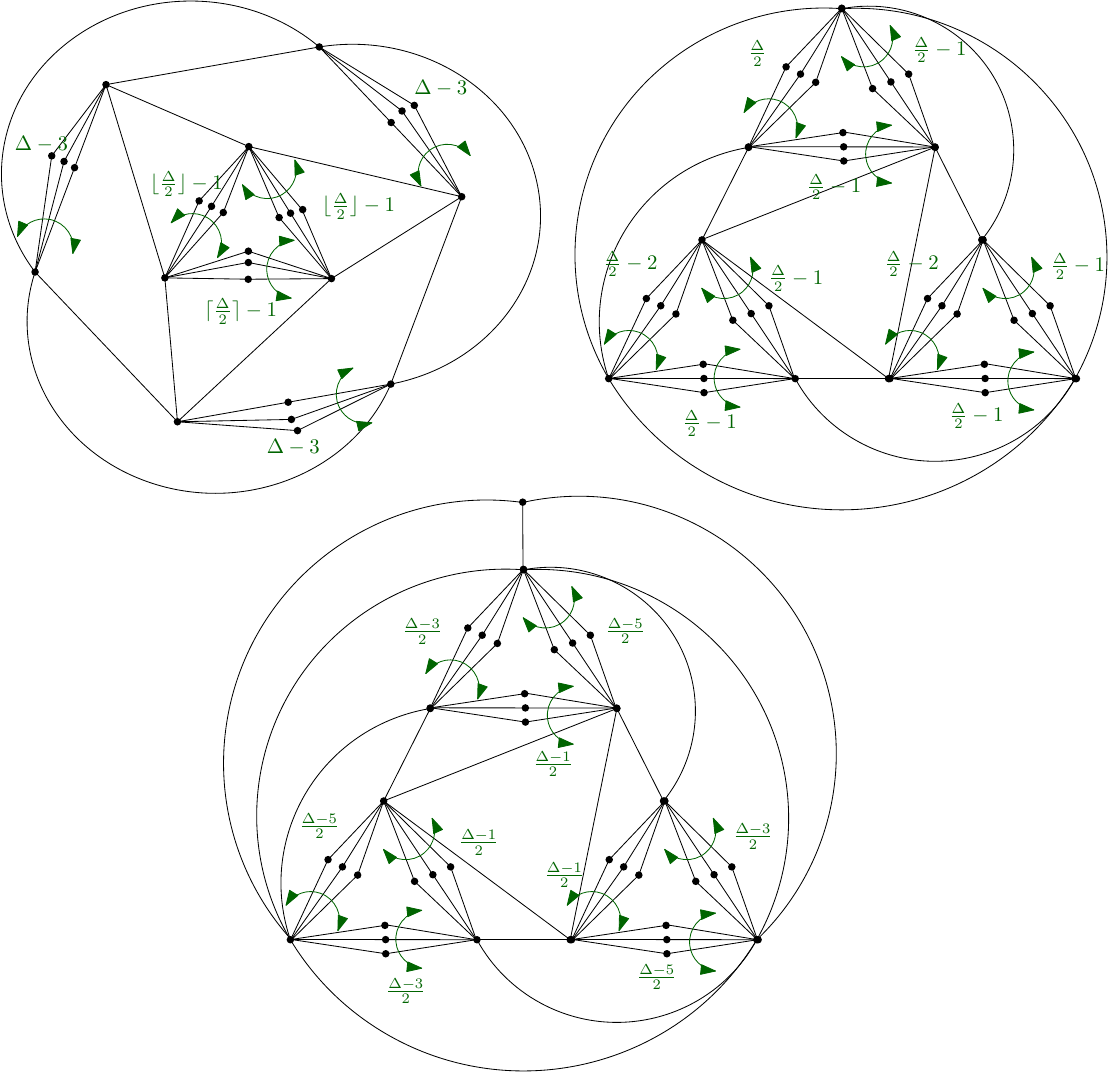}
    \caption{Constructions of planar graph with diameter $3$ and $\lfloor \tfrac92\Delta\rfloor -3$ vertices. The green values correspond to the number of twins of degree $2$ that are added. The construction on the top left is given in \cite{FHS95}. The construction on the top right also has order $\lfloor \tfrac92\Delta\rfloor -3$ when $\Delta$ is even, and the bottom construction has order $\lfloor \tfrac92\Delta\rfloor -3$ when $\Delta$ is odd.}
    \label{fig: lower-bound}
\end{figure}

Dorfling, Goddard and Henning \cite{DGH06} proved that every sufficiently large planar graph with diameter $3$ has a dominating set of size $6$, implying that when $\Delta$ is large enough, we have $\np_{\Delta,3}\leq 6\Delta + 6$. Moreover, we believe that using some results and techniques from \cite{DGH06}, one can improve this upper bound and derive a bound of the form $5\Delta + O(1)$. 

The degree diameter problem has also been studied on the classes of planar graphs having some restrictions on the possible lengths of facial cycles. In particular, it has been solved for the classes of planar quadrangulations of diameter at most $3$ \cite{DHS16} and of planar pentagulations of diameter at most $3$ \cite{DP24}. Moreover, in  \cite[Section 9]{DP24}, the author raises the question about planar triangulations of diameter at most $3$.
We also refer to \cite{Wiki} for more results on the degree diameter problem in planar graphs.
 
\paragraph*{Our results.}
Our main result states that the lower bound from \Cref{thm:previous_general_bound} is optimal, up to an additive constant.

\begin{theorem}
     \label{thm: main}
     For every $\Delta$, we have 
     $$\np_{\Delta, 3}\leq \frac{9}{2} \Delta + 9 + 39^3.$$ 
\end{theorem}

We also provide a new construction of planar graphs of diameter 3 with $\lfloor\frac{9}{2}\Delta\rfloor-3$ vertices for any $\Delta\geq 5$ (see \Cref{fig: lower-bound}). Note also that both graphs from \Cref{fig: lower-bound} can be triangulated, up to making the maximum degree increase by $1$.  Therefore, \Cref{thm: main} also partially answers the above mentioned question of~\cite[Section 9]{DP24} as it provide an upper bound which is optimal, up to an additive constant.

\paragraph*{Fractional matchings of neighbouring sets of edges.}
Our proof of \Cref{thm: main} relies on a reduction to a special instance of the fractional matching problem in planar graphs. If $G$ is a graph, we say that a subset $R$ of edges from $E(G)$ is a \emph{neighbouring set} if for every two edges $e,f\in R$, there exist respective endvertices $u,v$ of $e$ and $f$ such that $d_G(u,v)\leq 1$. For every subset $R\subseteq E(G)$, we let $G[R]$ denote the subgraph of $G$ induced by the set of edges $R$, that is, the vertices of $G[R]$ are exactly the set $V_R$ of endvertices of edges in $R$, and the edges of $G[R]$ are the edges from $R$. 
We show that for every neighbouring set $R$ of edges in a planar graph $G$, the subgraph $G[R]$ of $G$ induced by the edges of $R$ has fractional matching number $\mu^{\ast}$ at most $\tfrac{9}{2}$. In fact, we prove this result in more generality when $G$ is $K_5$-minor-free (and not necessarily planar).

\begin{restatable}{theorem}{fractionalMatching}
    \label{thm: frac-intro}
    Let $G$ be a $K_5$-minor-free graph and $R\subseteq E(G)$ a neighbouring set of edges. We have:
    $$\mu^{\ast}(G[R])\leq \frac{9}{2}.$$
\end{restatable}

Moreover, with a small amount of additional work, we can deduce from our proof a complete characterization of neighbouring sets of edges that attain the bound $\tfrac92$.

\begin{restatable}{theorem}{corollaryOfMatching}
    \label{cor: minimal_92}
    Let $G$ be a planar graph and $R\subseteq E(G)$ be a neighbouring set of edges. Then $\mu^{\ast}(G[R])=\tfrac92$ if and only if $G[R]$ contains a subgraph isomorphic to one of the following graphs: 
 $$3K_3, ~K_3+3K_2, ~C_7+K_2.$$
\end{restatable}

Neighbouring sets of edges naturally appear when one tries to construct large graphs with diameter $3$: assume that $R$ is a neighbouring set of edges of some graph $G$. Note that as $R$ is neighbouring, the graph $G'$ obtained from $G[R]$ when adding all edges in $E(G)\setminus R$ connecting two different edges from $R$ has diameter $3$. Moreover, note that any graph obtained after adding to $G'$ any number of degree $2$ vertices attached to both endvertices of edges of $R$ still has diameter at most $3$. For example, the graph from \Cref{fig: lower-bound} (top left) can be obtained this way, when starting from the graph $K_3+3K_2$ which is the disjoint union of a triangle and a matching of size $3$ (to see that this graph is a neighbouring set of edges of a planar graph, see \Cref{fig: minimal_92}). If we want to construct large graphs with maximum degree at most $\Delta$ and diameter $3$, note that the number of vertices of degree $2$ that we can add using this technique is upper bounded by $\mu^{\ast}(G[R])\Delta$. From this perspective, \Cref{thm: frac-intro} implies that all planar graphs that can be constructed using this technique have order at most $\tfrac92\Delta$.

Our proof of \Cref{thm: main} shows in some way that every large planar graph with diameter at most $3$ is ``close'' to a graph that can be obtained from this construction.

\paragraph*{Overview of the proof of Theorem \ref{thm: main}.}
To prove Theorem \ref{thm: main}, we show that,
given a planar graph $G$ of diameter $3$, we can construct an auxiliary planar graph $\Gamma$ with a neighbouring set $R\subseteq E(\Gamma)$ of edges such that $|V(G)|\leq \mu^{\ast}(\Gamma[R])\cdot \Delta + c$ for some constant $c\leq 9+39^3$ (see Theorem \ref{thm: main1}). The strategy to construct $\Gamma$ is the following: we start setting $G_0:=\Gamma_0:=G$, $R_0:=\emptyset$, and at step $i\geq 1$,
we look for a cycle $C_i$ of size at most $6$ in $G$, such that:
\begin{itemize}
 \item each $C_i$ is a separating cycle in $G$, i.e., both faces of $C_i$ contain a vertex of $G$;
 \item there exist two distinct vertices $u_i,v_i\in V(C_i)$ such that the set $\sg{u_i,v_i}$ dominates one of the regions $F_i$ from $\mathbb R^2$ delimited by $C_i$. 
\end{itemize}
If at step $i$ we can find such a cycle $C_i$, then we consider the graphs $\Gip, \Gamip$ respectively obtained from $\Gi, \Gami$ after removing every vertex lying in the interior of the dominated face $F_i$ (note that in our proof, we will in fact need some additional hypothesis on $C_i$, as we will need to ensure that the graph $\Gip$ still has diameter $3$). 
Note that then, the cycle $C_i$ is facial in both $\Gip$ and $\Gamip$. We moreover add in $\Gamip$ a (potentially) new edge $u_iv_i$, and 
two new pendant edges $u_ia_i, v_ib_i$. We then let $R_{i+1}$ be the set obtained from $R_i$ after adding the edge $u_iv_i$, and potentially some of the edges from $\sg{u_ia_i, v_ib_i}$,
according to whether $u_i$ and $v_i$ had private neighbours from $\Gi$ in the region $F_{i}$ or not (see \Cref{fig: detailed-ex,fig: detailed-ex2}). At each step, we will define a fractional matching $\mu_i$ on $\Gamma_i[R_i]$, such that every vertex lying in the interior of $F_i$ contributes to $\tfrac{1}{\Delta(G)}$ in the total value of $\mu_i$. This way, we will ensure that the total number $|V(G)\setminus V(G_i)|$ of vertices that have been removed between step $1$ and step $i$ is at most $\mu^{\ast}(\Gamma_i[R_i])\cdot\Delta(G)$. We will moreover prove that the set $R_i$ we define is a neighbouring set of edges in $\Gami$. Eventually, we will show that if at some point, we cannot find anymore a cycle $C_i$ with the aforementioned properties, then, 
up to adding some additional set of pendant edges $R$ to $\Gamma_i$ (resulting in a graph $\Gamma$) and removing from $\Gi$ some well chosen set of vertices $X\subseteq V(\Gi)$, we can extend $\mu_i$ to a fractional matching $\mu$ of $\Gamma[R_i\cup R]$ such that each vertex from $X$ contributes to $\tfrac{1}{\Delta(G)}$ in the total value of $\mu$. In particular, it implies that 
$|(V(G)\setminus V(G_i))\uplus X|\leq \mu^*(\Gamma[R_i\cup R])$. 
To complete the proof, we will show that $R_i\cup R$ is a neighbouring set of edges in $\Gamma$, and that we can choose $X$ such that the graph $G_i- X$ has at most $9+39^3$ vertices. To show this last point, we will use a result of MacGillivray and Seyffarth \cite{MS96} (later improved by Dorfling, Goddard and Henning \cite[Theorem~3]{DGH06}) showing that planar graphs of diameter $3$ have bounded domination number. Combining the above results with \Cref{thm: frac-intro}, we then deduce a proof of \Cref{thm: main}.

We give an illustration of the different steps of our proof on two examples in \Cref{fig: detailed-ex,fig: detailed-ex2}. The example from \Cref{fig: detailed-ex2} shows a typical situation in which we cannot stop our proof after the first step (emptying the cycles $C_i$), and where we need to consider the additional set $R$ of pendant edges described above.

\begin{figure}[htb]
  \centering   
  \includegraphics[scale=0.9]{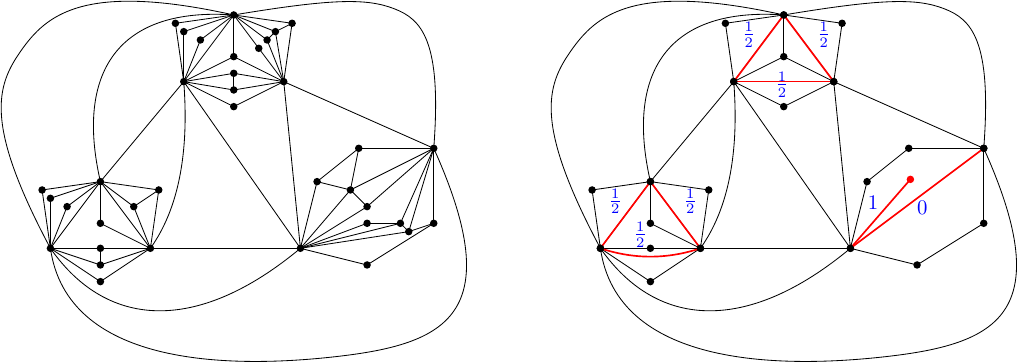}
  \caption{Left: A planar graph $G$ with diameter $3$. Right: The planar graph $\Gamma$ obtained after emptying iteratively some cycles of length at most $6$ that contain two vertices dominating one of the two faces of the plane they separate. We depicted in red the elements of $R$ that form a neighbouring set of edges in $\Gamma$, and in blue an optimal fractional matching of $\Gamma[R]$, with value $4\leq \tfrac92$.}   
  \label{fig: detailed-ex}
\end{figure} 

\begin{figure}[htb]
  \centering   
  \includegraphics[scale=0.8]{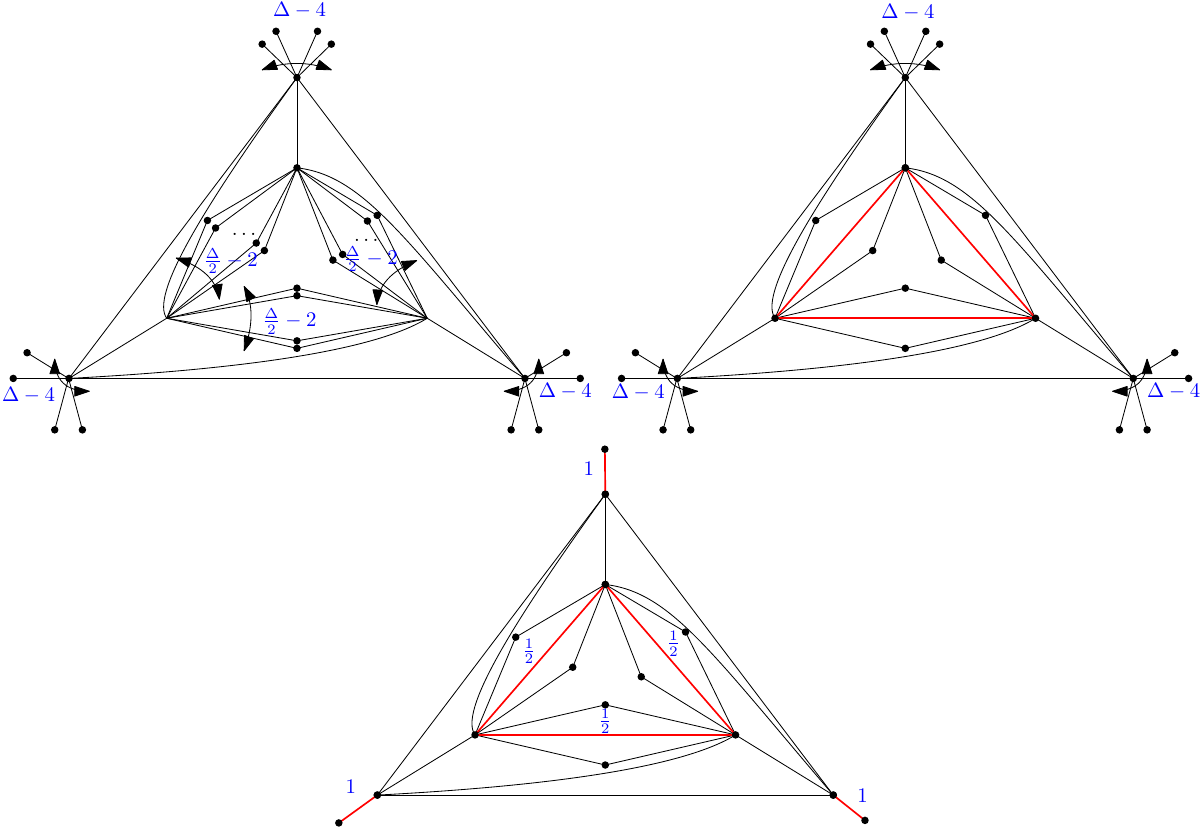}
  \caption{Top left: A planar graph $G$ with diameter $3$.
  Top right: The planar graph $\Gamma_i$ obtained after
  the first step of our proof, i.e., after emptying
  iteratively some cycles of size at most $6$ that contain
  two vertices dominating one of the two regions of the
  plane they separate. Bottom: The graph $\Gamma$ obtained
  after removing also vertices from $X$, and adding some
  more pendant edges to $\Gamma_i$. We depicted in red the
  elements of $R$ that form a neighbouring set of edges in
  $\Gamma$, and in blue an optimal fractional matching of 
  $\Gamma[R]$, with value $\tfrac92$.}   
  \label{fig: detailed-ex2}
\end{figure} 

\paragraph*{Structure of the paper.}

The sequence of cycles $(C_i)_i$ described above will be obtained by considering boundaries of some special subgraphs of $G$ called \emph{lanterns}, satisfying some specific properties. We present in \Cref{sec:empty_lanterns} a number of notions and useful properties about lanterns. The construction of the graph $\Gamma$, together with a proof of all the aforementioned properties is given in \Cref{sec: reduction}, which is dedicated to the proof of \Cref{thm: main1}. Eventually, we give a proof of Theorems \ref{thm: frac-intro} and \ref{cor: minimal_92} in \Cref{sec: fractional-matching}, whose combination with \Cref{thm: main1} immediately implies a proof of \Cref{thm: main}. Note that \Cref{sec: fractional-matching} is completely independent from the previous sections, and we believe that its content is of independent interest.

\section{Preliminaries}
\label{sec: def}

\subsection{Basic definitions and notations}
The graphs considered in this paper are simple and without loops. Moreover we will focus on planar graphs, that is, graphs which can be drawn on the plane without edge-crossings. The \textit{degree} of a vertex is the number of edges incident to it. Given a graph $G=(V,E)$, the \textit{maximum degree} of a graph $G$ is the maximum degree of its vertices and is denoted with $\Delta(G)$. When no ambiguity is possible, we will denote it by $\Delta$. The \textit{order} of $G$ corresponds to the number $|V(G)|$ of vertices of $G$.

A \emph{path} in $G$ is a sequence $u_0,u_1,\ldots, u_k$ of vertices such that for each $0\leq i<k$, $u_iu_{i+1}$ is an edge; $u_0$ and $u_k$ are called its \emph{extremities} while $u_1,\ldots,u_{k-1}$ are called its \emph{internal vertices}; its \emph{length} is $k$, that is, the length of a path corresponds to its number of edges. A path with extremities $x$ and $y$ will also be called a \emph{$xy$-path}.
We say that two paths are \emph{disjoint} (resp. \emph{internally disjoint}) if they have no common vertices (resp. internal vertices).

The \textit{distance} between two vertices $x$ and $y$ of $G$, denoted by $d_G(x,y)$, is minimum length (i.e. number of edges) of a shortest path between $x$ and $y$ in $G$. If there is no such path between $x$ and $y$, then we set $d_G(x,y):=\infty$. Similarly, given two sets $X,Y$ of vertices of $G$, the distance between $X$ and $Y$ in $G$ is $d_G(X,Y)=\min\{d_G(x,y) ; x\in X, y\in Y\}$.
The \textit{diameter} of $G$ is the maximum over the distances between two vertices of $G$. The \textit{eccentricity} of a vertex $x$ is the maximum over the distances between $x$ and another vertex of the graph, and the \textit{radius} of $G$ is the minimum over the eccentricities of its vertices.

A \textit{dominating set} of $G$ is a subset $D$ of $V(G)$ such that every vertex of $G$ is either in $D$ or adjacent to a vertex from $D$. The \textit{domination number}, denoted $\gamma(G)$  is the minimum cardinality of a dominating set of $G$.

Let $G$ be a graph. Given a subset $X$ of vertices of $G$, we let $G[X]$ denote the subgraph of $G$ induced by the vertices from $X$. For each $R\subseteq E(G)$, we also let $G[R]$ denote the edge-induced subgraph of $G$, that is, the subgraph obtained after removing all the edges from $G$ that do not belong to $R$, and all the vertices which are not incident to some edge of $R$. 
A set $R\subseteq E(G)$ of edges in $G$ is called \emph{neighbouring} if every two edges $e,e'$ from $R$ are at distance at most $1$ in $G$, that is, either $e$ and $e'$ share a common endvertex, or $e$ and $e'$ are both adjacent to some edge $f \in E(G)$.

If $v$ is a vertex of $G$, then we denote with $N_{G}(v)$ its neighbourhood in $G$. 
Given a graph $G$, and a set $X\subseteq V(G)$, we let $N_G(X):=\cup_{x\in X}N_G(x)$ denote the neighbourhood of $X$ in $G$.
With a slight abuse of language, we say that two sets of vertices $X,Y$ are adjacent in $G$ there exists $x\in X$ and $y\in Y$ such that $x$ and $y$ are adjacent in $G$. In particular, as we consider edges of $G$ as subsets of size $2$ of $V(G)$, we will also allow to say that two edges are adjacents.

A graph $H$ is a \emph{proper} subgraph of $G$ if $H$ is a subgraph of $G$ that is not isomorphic to $G$. We say that $H$ is an \emph{isometric subgraph} of $G$ if $H$ is a subgraph
of $G$, and moreover, for all $x,y\in V(H)$, we have $d_H(x,y)=d_G(x,y)$.

For every $k\geq 1$, we let $K_k$ denote the clique on $k$ vertices and $C_k$ denote the cycle on $k$ vertices. For every two graphs $G,H$, we let $G+H$ denote their disjoint union. For every $n\in \mathbb N\setminus \sg{0}$, and every graph $G$, we let $n G$ denote the graph obtained by taking a disjoint union of $n$ copies of $G$.

\subsection{Planar graphs}
In the remainder of the paper, we will always assume that all planar graphs we consider come with a fixed embedding, \ie we work on \emph{plane graphs}. A graph $H$ is a \emph{plane} subgraph of a plane graph $G$ if $H$ is obtained from $G$ by deleting some vertices and edges while keeping the same embedding. 

A \emph{face} of a planar graph $G$ is an arcwise connected component of the open set $\mathbb{R}^2\setminus G$, and is thus an open subset of $\mathbb{R}^2$. We denote by $\overline{F}$ the closure of a face $F$ in $\mathbb{R}^2$, and the \emph{boundary} of a face $F$ of $G$ is the set of vertices and edges of $G$ drawn in $\overline{F}\setminus F$. If the boundary of a face $F$ of $G$ is a cycle $C$ of $G$, then we call $C$ the \emph{cycle boundary} of $F$ in $G$. If a cycle $C$ is the cycle boundary of a face of $G$, then we say that $C$ is \emph{facial} in $G$. By abuse of vocabulary, we will say that a face $F$ contains a subgraph $H$ of $G$, if $H$ is drawn in $F$.

\begin{remark}
 \label{rem: cycle-dom} 
 For every plane graph $G$ of diameter at most $3$, every cycle $C$ of $G$ dominates one of its two faces, \ie one of the two arcwise connected components of $\mathbb{R}^2\setminus C$ is such that every vertex drawn inside it has at least one neighbour on $C$.
\end{remark}

\subsection{Fractional matchings and coverings}
\label{def: frac}

\paragraph*{Fractional matchings}
A \emph{matching} in a graph $G$ is a set of edges that do not pairwise share any endvertex. We denote with $\mu(G)$ the maximum size of a matching in $G$. A \emph{fractional matching} in a graph $G$ is a mapping $g: E(G)\to \mathbb R_{\geq 0}$ such that for each vertex $v\in V(G)$, we have 
$$\sum_{\substack{e\in E(G)\\ v\in e}}g(e)\leq 1.$$
The \emph{value} of $g$ is the sum $\sum_{e\in E(G)}g(e)$, and we denote with $\mu^{\ast}(G)$ the supremum over the values of all fractional matchings of $G$. As the fractional matching problem corresponds to the usual relaxation of the matching problem, we clearly have that for all graph, $\mu(G)\leq\mu^{\ast}(G)$. 

\paragraph*{Fractional vertex-covers}
A \emph{vertex-cover} in a graph $G$ is a set $X$ of vertices such that  
every edge is incident to at least one vertex of $X$. We denote with $\tau(G)$ the minimum size of a vertex cover in $G$. A \emph{fractional vertex-cover} in a graph $G$ is a mapping $h: V(G)\to \mathbb R_{\geq 0}$ such that for every edge $uv\in E(G)$, we have $h(u)+h(v)\geq 1$. 
The \emph{value} of $h$ is the sum $\sum_{v\in V(G)}h(v)$, and we denote with $\tau^{\ast}(G)$ the infimum over the values of all fractional vertex-covers of $G$. To simplify notation, for each set $X\subseteq V(G)$, we set $h(G):=\sum_{x\in X}h(x)$ (in particular the value of $h$ is $h(G)$).
As the fractional vertex-cover problem corresponds to the usual relaxation of the vertex cover problem, we have that for all graphs $\tau^{\ast}(G)\leq \tau(G)$. 

The fractional vertex-cover problem also corresponds to the dual of the fractional matching problem. Hence, a consequence of the Strong Duality Theorem in linear programming is that for every graph $G$, we have $\tau^{\ast}(G)=\mu^{\ast}(G)$.

\section{Lantern extraction}
\label{sec:empty_lanterns} 
In this section, we introduce vocabulary and  useful properties related to what we will call lanterns.

\paragraph{Lanterns}
A \emph{lantern} $L$ in a graph $G$ is a subgraph of $G$ containing two \emph{hubs}, that is, two distinct vertices $u,v\in V(G)$, together with a set of at least two pairwise internally disjoint $uv$-paths, 
called the \emph{axes} of $L$. The \emph{width} of $L$ is its number of
axes, and the \emph{length} of $L$ is the maximum length of an axis of $L$. A \emph{$k$-lantern} is a lantern of width $k$.
Given two lanterns $L,L'$, we say that $L'$ is a \emph{sublantern} of $L$ if it is a subgraph of $L$, and if moreover $L$ and $L'$ have the same hubs.

\begin{remark}
    \label{rem: theta}
    The term lantern was introduced in~\cite{GH02} and later reused in \cite{DGH06}. As we will need some results from these papers, we decided to stick to the vocabulary introduced there.
    Note that in structural graph theory, lanterns are also known as \emph{theta-graphs}.
\end{remark}

\paragraph{Dominating lanterns}
Note that for every $k\geq 2$, any $k$-lantern of a plane graph $G$ is a plane subgraph of $G$ with exactly $k$ faces, $k-1$ of them being bounded subsets of $\mathbb R^2$. In particular, if $L$ is a lantern with hubs $u,v$ and $F$ a face of $L$, such that the set of vertices of $G$ which do not belong to $F$ is dominated by $\sg{u,v}$, then we say that $L$ is a \emph{dominating lantern}. See \Cref{fig:dom-lantern}. In what follows, we will assume that every dominating lantern $L$ always comes with a fixed such special face $F$ (if there are multiple candidates for $F$, we just choose any of them). The face $F$ will be called the \emph{free face} of $L$, while its complementary face $\mathbb R^2\setminus \overline{F}$ will be called the \emph{interior} of $L$.
Note that if $L$ is a dominating lantern and if $L'$ is a sublantern of $L$ of width at least $2$, then $L'$ is also a dominating lantern, whose interior is included in the interior of $L$. We stress out that the interior of a lantern does not necessarily correspond to a bounded region of $\mathbb R^2$; this region will be bounded only when the free face is the external face of $L$ in the fixed embedding of $G$ (even though for simplicity, we will consider in all the figures given later that the free face is the unbounded one).

\begin{remark}
\label{rem: region}
    It is worth mentioning that the (topological) closure of the interior of a short dominating lantern corresponds exactly to the notion of \emph{region} of a plane graph, introduced in \cite{alber2004polynomial} in the context of parameterized algorithms for domination in planar graphs. 
\end{remark}

\begin{figure}[htb]
  \centering   
  \includegraphics[scale=1]{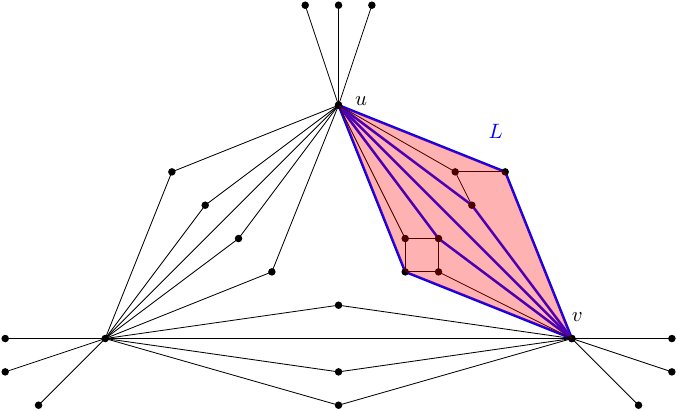}
  \caption{A short dominating $5$-lantern $L$ in a plane graph of diameter $3$. The axes of $L$ are colored with blue. In this embedding, the free face of $L$ is its unbounded face, and the interior of $L$ is colored with light red.}
  \label{fig:dom-lantern}
\end{figure}

In what follows we prove a series of lemmas allowing us to find lanterns with useful properties in bigger lanterns.

\begin{lemma}
 \label{lem: new-dom-lanterns}
 If $G$ is a plane graph of diameter at most $3$, then, for $k\geq 8$, every $(k+2)$-lantern of $G$ contains a dominating sublantern of width $k$.  
\end{lemma}

\begin{proof} 
 Let $k\geq 8$ and $L$ be a $(k+2)$-lantern in $G$ with hubs $u$ and $v$, and set $k' := k+2$. 
 We denote the axes of $L$ with $P_1, \ldots, P_{k'}$ so that, for every $i\in \sg{1, \ldots, k'}$, $P_i\cup P_{i+1}$ is a facial cycle of $L$, and denote its associated face with $F_i$ (indices are taken modulo $k'$). For each $i\in \sg{1,\ldots, k'}$, we say that the paths $P_i$ and $P_{i+1}$ are consecutive.

 Denote with $S$ the set of vertices of $G$ at distance at least 2 from both hubs of $L$. We let $I$ denote the set of indices $i\in \sg{1, \ldots, k'}$ such that $S$ intersects the face $F_i$. Note that if $|I|\leq 1$, then $L$ is a dominating lantern, thus we may assume that $|I|\geq 2$. As $G$ has diameter at most $3$, every two elements from $S$ are connected in $G$ by a path of length at most $3$, which, by definition of $S$ cannot contain $u$ or $v$. In particular, such a path can only intersect at most $4$ different paths $P_i$, which moreover must be consecutive. As we assumed that $k\geq 8$, this implies that 
 $S$ can only intersect the boundaries of at most $3$ consecutive faces $F_i$. We assume without loss of generality that such faces are included in $\sg{F_k,F_{k+1},F_{k+2}}$. In particular, note that the sublantern $L':=P_1\cup \cdots\cup P_{k}$ is a dominating lantern, whose interior is delimited by the cycle $P_1\cup P_k$. 
 \end{proof}

 Note that the result of \Cref{lem: new-dom-lanterns} is optimal in the sense that the lower bound $k\geq 8$ cannot be decreased, as shown in ~\Cref{fig:9lantern6sublantern}.

 \begin{figure}[!ht]
     \centering
     \includegraphics[scale=0.9]{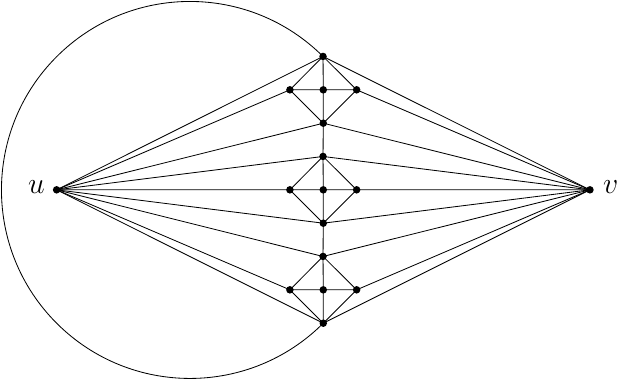}
     \caption{A 9-lantern containing no dominating sublantern of width 7.}
     \label{fig:9lantern6sublantern}
 \end{figure}

\paragraph{Chordless lanterns}
We say that an axis of a lantern is \emph{chordless} if it contains no edge between non-consecutive vertices, except possibly between the hubs of the lantern.
We say a lantern is chordless if all its axes are chordless.

\paragraph{Short lanterns}
We call a lantern \emph{short} if it has length at most $3$.
Note that a sublantern of a short (resp. chordless) lantern is also short (resp. chordless).

\begin{lemma}
 \label{lem: short-lanterns}
 If $G$ is a plane graph of diameter at most $3$, then, for every $k\geq 4$, if $G$ has a dominating $(k+2)$-lantern $L$, then it also contains a dominating short chordless lantern $L'$ of width $k$ having the same hubs as $L$, whose interior is included in the interior of $L$ and such that $V(L')\subseteq V(L)$.  
\end{lemma}

\begin{proof}
 Let $k':=k+2$ and $L$ be a dominating $k'$-lantern of $G$ with hubs $u$ and $v$ and axes $P_1, \ldots, P_{k'}$ such that for every $i\in \sg{1, \ldots, k'}$, $P_i\cup P_{i+1}$ is a facial cycle of $L$ (indices are taken modulo $k'$). Denote by $F$ the face of $L$ with boundary $C:=P_1\cup P_{k'}$, and assume without loss of generality that $F$ is the free face of $L$. 

 Consider an axis $P_i=(v_0:=u, v_1, v_2, \ldots, v_k:=v)$ of length at least $3$, for some $i\in \sg{2, \ldots, k+1}$. Suppose that $P_i$ is not chordless, i.e. that there exist $j,j'$ such that $j'>j+1$, $v_j$ and $v_{j'}$ are adjacent and $(j,j') \neq (0,k)$. Since $L$ is a dominating lantern, all the vertices between $v_j$ and $v_{j'}$ are adjacent to $u$ or $v$.
 Then, if we replace $P_i$ by 
 $(v_0,\ldots v_j, v_{j'}, \ldots, v_k)$ in $L$, we still obtain a dominating $k'$-lantern with free face $F$.
 
 Let us apply this operation as long as it is possible on each axis other than $P_1$ and $P_k$: we denote $L'$ the resulting dominating $k'$-lantern with free face $F$, and for each $i$, we denote by $P'_i$ the axis replacing $P_i$. Then for each $i\in \sg{2, \ldots, k+1}$, $P'_i$ is chordless and has length at most $3$. Indeed, if $P'_i$ has length at least $4$, we note $P_i=(u, v_1, v_2,v_3, \ldots, v_\ell:=v)$, then $v_2$ is adjacent to either $u$ or $v$, and we can apply the operation again. 
 In particular, the sublantern $L''$ obtained from $L'$ after taking the union of all the axes different from $P_1$ and $P_k$ forms a short chordless dominating lantern of width $k$, with hubs $u,v$, and whose interior is included in the interior of $L$. 
\end{proof}

\paragraph{Hub-faithful lanterns }
A lantern $L$ of a planar graph $G$ is called \emph{hub-faithful} if it is dominating, and moreover there exists a shortest path in $G$ between its two hubs that does not intersect the interior of $L$. In particular, if the two hubs of $L$ are adjacent, then the edge between them must be drawn in the closure of the free face of $L$.

\begin{lemma}
\label{lem: emptying}
 Let $G$ be a plane graph of diameter at most $3$, $L$ be a short hub-faithful $k$-lantern with $k\geq 4$, and $F$ be the interior of $L$. Denote by $G'$ the induced plane subgraph of $G$ obtained after removing every vertex drawn in the interior of $F$.
 Then, $G'$ is an isometric subgraph of $G$.
 Note that in particular, every sublantern of a hub-faithful lantern is also hub-faithful.
\end{lemma}

\begin{proof}
 We let $P, Q$ denote the two axes in $L$ such that the cycle $C:=P\cup Q$ forms the boundary of $F$. As both $P$ and $Q$ have length at most $3$, and as $L$ is hub-faithful it is not hard to check that for every pair of vertices $x,y\in V(P)$, we have $d_G(x,y)=d_{G'}(x,y)$. The same also holds for every pair of vertices $x,y\in V(Q)$. 
 Moreover, note that as $k\geq 4$, every path between an internal vertex of $P$ and an internal vertex of $Q$ of length at most $2$ must belong to $G'$. As $L$ is short, $C$ is a cycle of order at most $6$, so it implies that for every two vertices $x,y\in V(C)$, we have 
 $d_G(x,y)=d_{G'}(x,y)$. As $C$ separates $V(G')\setminus V(C)$ from $V(G)\setminus V(G')$, we can then conclude that for every two vertices $x,y\in V(G')$, we have $d_G(x,y)=d_{G'}(x,y)$, so $G'$ is an isometric subgraph of $G$, as desired.
\end{proof}

\begin{lemma}
 \label{lem: faithful}
 Let $G$ be a plane graph of diameter at most $3$, $k\geq 2$ and $L$ be a short chordless dominating $(2k-1)$-lantern of $G$. Then $L$ contains a sublantern of width $k$ which is hub-faithful.
\end{lemma}

\begin{proof}
Let $u,v$ denote the hubs of $L$. We let $F$ denote the interior of $L$, with cycle boundary $C$. Assume that $L$ is not hub-faithful, and consider a shortest path $uv$-path $P$ which intersects $F$. Note that as $L$ is short, $P$ must have length at most $2$, and thus in particular it is entirely contained in $F$. 

Let $P_1, \ldots, P_{k'}$ denote the axes of $L$, where $k':=2k-1$, such that for each $i\in \sg{1,\ldots, k'}$, $P_i\cup P_{i+1}$ is the cycle boundary of a face $F_i$ of $L$ (indices are taken modulo $k'$), and such that $C=P_1\cup P_{k'}$. As $P$ has length at most $2$, it is either contained in some face $F_i$, or equal to a path $P_i$ (the latter follows from the fact that $L$ is chordless). Up to replacing each index $i'$ by $k'+1-i'$, we may assume without loss of generality that $i\leq k$. In particular, it then implies that the path $P$ does not intersect the interior of the sublantern $L':=P_{k}\cup\ldots\cup P_{k'}$, which has width $k$, and whose free face has cycle boundary $P_{k}\cup P_{k'}$. Hence $L'$ satisfies the desired properties.
\end{proof}

\paragraph{Deep lanterns}
Let $L$ be a dominating lantern in a plane graph $G$, and let $u,v$ denote its hubs. We
let $F$ denote the interior of $L$, and $C$ its boundary cycle.
We let $U_L$ (resp. $V_L$) be the set of vertices of $V(G)\setminus V(C)$ drawn in $F$ that are adjacent to $u$ but not to $v$ (resp. to $v$ but not to $u$). In the remainder of the paper, we will always denote the hubs of lanterns with the letters $u$ and $v$ (up to adding some indices/superscript), in order to be able to reuse the notations $U_L, V_L$ without any ambiguity.
We say that $L$ is \emph{deep} if it satisfies the following two conditions
 
\begin{itemize}
 \item If $U_L\neq \emptyset$, then there exists a vertex $u'\in U_L$ such that
 for every $w\in V(C)\setminus \sg{u,v}$, $d_G(u', w)=d_{G[V(C)]+uu'}(u',w)$\footnote{Recall that $G[V(C)]+uu'$ is obtained by adding to the graph induced by $C$ the vertex $u'$ and the edge $uu'$.}. In particular, $u$ is the only neighbour of $u'$ in $C$.
 \item If $V_L\neq \emptyset$, then there exists a vertex $v'\in V_L$ such that 
 for every $w\in V(C)\setminus \sg{u,v}$, $d_G(v', w)=d_{G[V(C)]+vv'}(v',w)$. In particular, $v$ is the only neighbour of $v'$ in $C$.
\end{itemize}

See \Cref{fig: deep} for an illustration.

\begin{figure}[htb]
  \centering  
  \includegraphics[scale=1]{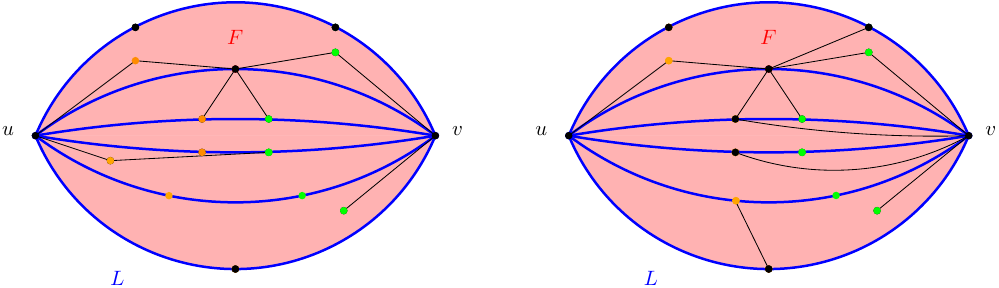}
  \caption{In both drawings, a short dominating lantern $L$ is represented in blue. The vertices from $U_L$ and $V_L$ are respectively colored in orange and in green. The face $F$ is colored in light red. 
  In the left drawing, $L$ is deep, while 
  it is not the case anymore in the right drawing}
  \label{fig: deep}
\end{figure}

\begin{lemma}
 \label{lem: deep}
 Let $G$ be a plane graph of diameter at most $3$, $k\geq 2$ and $L$ be a short hub-faithful $(k+12)$-lantern of $G$. Then $L$ contains a hub-faithful $k$-sublantern which is deep.
\end{lemma}

\begin{proof}
 Let $u,v$ denote the hubs of $L$, $k':=k+12$ and let $F$ denote the interior of $L$. We let $P_1, \ldots, P_{k'}$ denote the axes of $L$ such that for every $i\in \sg{1,\ldots, k'}$, $P_i$ and $P_{i+1}$ form together a facial cycle of $L$ (indices are taken modulo $k'$). Moreover, we assume that $C:=P_1\cup P_{k'}$ forms the boundary cycle of $F$.
 For every $i\leq j \in \sg{1,\ldots, k'}$, we let $L_{i,j}$ denote the sublantern of $L$ obtained after taking the union of the axes $P_l$, for $l\in \sg{i, i+1, \ldots, j}$. We set $L':=L_{4, k'-3}$ and $L'':=L_{7, k'-6}$. We also let $F'$ and $F''$ be respectively the interiors of $L'$ and $L''$, so that $\overline{F''}\subseteq \overline{F'}\subseteq \overline{F}$ (see \Cref{fig: deep-proof}). Note that we then have $U_{L''}\subseteq U_{L'}\subseteq U_L$, and 
 $V_{L''}\subseteq V_{L'}\subseteq V_L$.
 
 Observe that as $L$ is short, every vertex $u'\in U_{L'}$ (resp. $v'\in V_{L'}$) satisfies that
 for every $w\in V(C)\setminus \sg{u,v}$,
 we have $d_G(u', w)=d_{C+uu'}(u',w)$ (resp. $d_G(v', w)=d_{C+vv'}(v',w)$). In particular, if both $U_{L'}$ and $V_{L'}$ are nonempty, then $L$ is a deep lantern and we are done.
 Note also that if $U_{L'}=V_{L'}=\emptyset$, then $L'$ is deep, so we may assume without loss of generality that $U_{L'}\neq \emptyset$ and that $V_{L'}=\emptyset$. 
 
 The same reasoning as above (with $L'$ and $L''$ playing respectively the roles of $L$ and $L'$) implies that if $U_{L''}$ is nonempty, then $L'$ is a deep lantern. Thus we may assume that $U_{L''}=\emptyset$. In particular, we then have $U_{L''}=V_{L''}=\emptyset$, implying that $L''$ is deep. 
\end{proof}

\begin{figure}[htb]
  \centering  
  \includegraphics[scale=0.8]{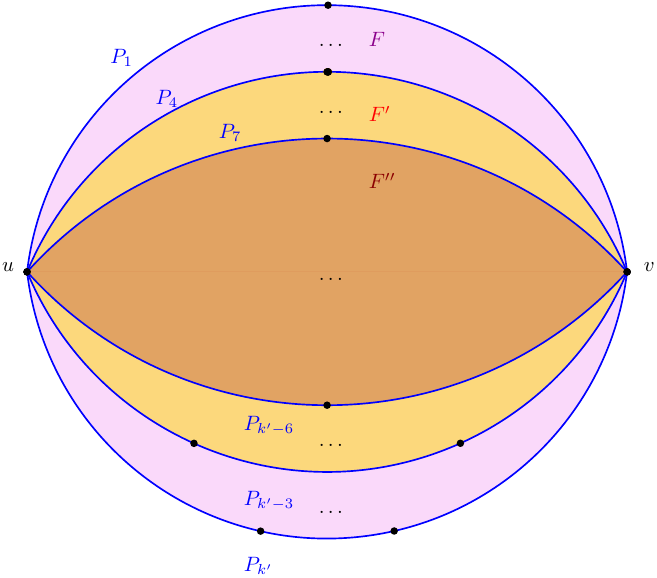}
  \caption{Proof of \Cref{lem: deep}.} 
  \label{fig: deep-proof}
\end{figure}

\paragraph{Nice lanterns}
We say that a lantern is \emph{nice} if it is hub-faithful (and thus also dominating), short, chordless, deep, and has width at least $6$.

\begin{lemma}
 \label{lem: extract}
 Let $G$ be a planar graph of diameter at most $3$ and $L$ be a $39$-lantern in $G$. Then $G$ contains a nice lantern $L'$ with the same hubs than $L$, and such that $V(L')\subseteq V(L)$.
\end{lemma}
\begin{proof}
 The proof immediately follows from consecutive applications of \Cref{lem: new-dom-lanterns,lem: short-lanterns,lem: faithful,lem: deep}.
\end{proof}

\section{Reduction to fractional matching}
\label{sec: reduction}
Our main result in this section is the following theorem, which establishes a connection between the degree diameter problem in planar graphs of diameter $3$, and fractional matchings. 

\begin{theorem}
 \label{thm: main1}
 Let $G$ be a planar graph with maximum degree $\Delta$ and diameter at most $3$. Then there exists a planar graph $\Gamma$ and a neighbouring set of edges $R\subseteq E(\Gamma)$ such that 
 $$|V(G)|\leq \mu^{\ast}(\Gamma[R])\cdot \Delta + 9 + 39^3.$$
\end{theorem}

Observe that combining \Cref{thm: main1} with \Cref{thm: frac-intro} immediately gives a proof of \Cref{thm: main}.

We split the proof of \Cref{thm: main1} in two parts, respectively described in \Cref{sec: vidage-1,sec: vidage-2}. 
In the whole section, we fix an integer $\Delta\geq 3$, together with a plane graph $G$ of diameter at most $3$ and maximum degree $\Delta$, with a fixed planar embedding. Again, every subgraph of $G$ will always be implicitly embedded in $\mathbb R^2$, with respect to the drawing of $G$.

\subsection{Emptying the lanterns}
\label{sec: vidage-1}
As sketched in the introduction, the first part of our proof of \Cref{thm: main1} consists in iteratively removing vertices lying  in the dominated face of some well-chosen cycles. We will find such cycles as boundaries of interiors of nice lanterns. 

\paragraph{Proof overview.}
\setcounter{footnote}{0}
We construct iteratively a sequence of graphs $G_0, \ldots, G_\ell$, starting with $G_0:=G$ such that for each $i\in \sg{0, \ldots, \ell-1}$, $\Gip$ is a proper isometric subgraph of $G_i$. Each graph $G_i$ will be considered with respect to the embedding of $G$. At step $i\geq 0$, we will consider a nice lantern $L_i$ and remove every vertex drawn in its interior to obtain $\Gip$. In parallel, we will also construct inductively a sequence $(\Gami, \Ri, \mui)_i$ such that for each $i\geq 0$, $\Gi$ is a subgraph of $\Gami$, $\Ri$ is a neighbouring set of edges of $\Gami$, and $\mui$ a fractional matching of $\Gami[\Ri]$ such that
$\left|V(G)\setminus V(\Gi)\right|= \left(\sum_{e\in \Ri}\mui(e)\right)\Delta$. 
We iterate this procedure and construct a new graph $\Gip$ as long $\Gi$ contains a $39$-lantern, and thus, by \Cref{lem: extract}, a nice lantern\footnote{In fact, for technical reasons, we will more precisely look for a lantern which is nice in $G$, and which is a subgraph of $\Gi$.}.

\paragraph{Construction of the sequence $\mathbf{(G_i, \Gamma_i, R_i, \mu_i)_{0\leq i\leq \ell}}$} 

Our goal now is to construct by induction over increasing values of $i\geq 0$ a sequence $(G_i, \Gamma_i, R_i, \mu_i)_i$ satisfying the properties mentioned above.

\begin{remark}
 \label{rem: preliminary}
 Note that if $G$ contains a lantern with hubs $u,v$ such that $\sg{u,v}$ dominates all vertices of $G$, then $|V(G)|\leq 2(\Delta(G)+1)$. Thus we will assume without loss of generality that $G$ has no such lantern. In particular, it implies that for every dominating lantern $L$, there is only one possible choice for the interior of $L$, and the free face is not dominated by $\sg{u,v}$.
\end{remark}

\paragraph*{Construction of $\mathbf{(\Gi)_i}$}
We first explain how to construct the sequence $(G_i)_{i\geq 0}$ of isometric subgraphs of $G$. 

\begin{claim}\label{clm: ipropGip}
    There exists a (finite) sequence $(G_i)_{0\leq i\leq \ell}$ such that $G_0=G$, $G_{\ell}$ has no $39$-lantern, and such that for each $i<\ell$, $\Gip$ is a proper isometric subgraph of $\Gi$, obtained by removing every vertex of $\Gi$ lying in the interior of a lantern $L_i$ of $\Gi$ with hubs $u_i, v_i$, which is nice in $G$. 
    
    In particular, each $G_i$ is a plane graph with maximum degree $\Delta$ and diameter at most $3$. 
   
    Moreover, if $F_i$ denote the interior of $L_i$, then we can construct $(G_i)_i$ such that for every $i\neq j$, $F_i$ and $F_j$ are disjoint.
\end{claim}

\begin{remark}
 \label{rem: region-decomposition}
 Note that a collection of (topological) closures of interiors of short dominating lanterns that are pairwise disjoint has also been considered and called \emph{region-decomposition} in \cite{alber2004polynomial}, and in subsequent related work on parameterized algorithms for planar graphs.
\end{remark}

\begin{proof}
 We construct the sequence $(G_i)_i$ by induction on $i\geq 0$. Assume that for some $i\geq 0$, a sequence $(G_j)_{0\leq j\leq i}$ with the desired properties has already been constructed. We moreover assume by induction that for every $j<i$, the interior $F_j$ of $L_j$ is not a proper subset of any interior of a nice lantern of $G$.
 
 In particular, an immediate induction implies that all graphs $G_j$ are isometric subgraphs of $G$.
 If $G_i$ has no $39$-lantern, then we conclude by setting $\ell:=i$.
 Assume now that $\Gi$ has a $39$-lantern $L^0_i$. 
 
 First, we claim that $\Gi$ has a lantern $L_i$ which is a nice lantern when considered as a lantern in $G$. To see this, note that $L^0_i$ is also a lantern in $G$, thus, by \Cref{lem: extract}, $G$ contains a lantern $L_i$ whose hubs are the same than the ones of $L^0_i$, which is nice in $G$, and such that $V(L_i)\subseteq V(L^0_i)$. In particular, as $\Gi$ is a subgraph of $G$, it implies that $L_i$ is also a lantern in $\Gi$.
 Among all possible choices, we now choose a lantern $L_i$ in $\Gi$, which is a nice lantern when considered as a lantern in $G$, and such that its interior $F_i$ is inclusionwise maximal. Note that such a choice makes sense, as we assumed in \Cref{rem: preliminary} that the free face of a lantern of $G$ cannot be dominated by its hubs. We denote by $u_i, v_i$ the hubs of $L_i$.

 Let $\Gip$ be the graph obtained from $\Gi$ after removing every vertex of $\Gi$ that lies in $F_i$. Note that as $L_i$ is an hub-faithful lantern in $G$, there exists a shortest path between $u_i$ and $v_i$ in $G$ which does not intersect $F_i$. In particular, as $\Gi$ is an isometric subgraph of $G$ that contains $u_i, v_i$, there also exists a shortest path between $u_i$ and $v_i$ in $\Gi$ which does not intersect $F_i$. It thus implies that $L_i$ is also hub-faithful as a lantern of $\Gi$. Thus, by \Cref{lem: emptying}, $\Gip$ is an isometric subgraph of $\Gi$. Moreover, as $L_i$ is nice in $G$, it has width at least $6$, so in particular $\Gip$ has strictly less vertices than $\Gi$. This shows that the sequence $(G_i)_i$ must be finite.
 
 To conclude the proof, it remains to show that the interiors $F_i$ are pairwise disjoint. It follows from a simple induction on $i\geq 0$: assume by that for some $1\leq i <\ell$, the interiors $F_0, \ldots, F_{i-1}$ constructed so far are pairwise disjoint. Then, as no vertex of $\Gi$ belongs to a face $F_j$ with $j<i$, if the face $F_i$ intersects such an interior $F_j$, then $F_j$ should be included in one of the faces of $L_i$ which is drawn in $F_i$, and thus we should have $F_j\subseteq F_i$. In particular, it contradicts the fact that, when constructing $F_j$, we chose $L_j$ such that $F_j$ was inclusionwise maximal.
\end{proof}

\begin{remark}
 \label{rem: degree-1}
 Note that, up to changing the initial embedding of $G$, we may assume moreover in the proof of \Cref{clm: ipropGip} that if for some $i<\ell$, some vertex with degree $1$ in $\Gi$ is adjacent to $u_i$ or $v_i$, then it is drawn in the interior of $L_i$, and is thus not a vertex of $\Gip$. This follows from the easy observation that every vertex of degree $1$ in some $\Gi$ must also have degree $1$ in $G$, as the operation of removing all vertices lying in the interior of a dominating lantern does not create new pendant edges.
\end{remark}

\paragraph*{Construction of $\mathbf{(\Gami,\Ri)_i}$ and additional properties}
From now on, we denote by $(G_i)_{0\leq i\leq \ell}$ the sequence of graphs given by \Cref{clm: ipropGip}, and we will reuse the notations $L_i, F_i, u_i, v_i$ from the statement.

We begin by defining the sequence $(\Gami)_{0\leq i\leq \ell}$ of auxiliary planar graphs. We set $\Gamma_0:=G$, and for each $i\geq 1$, we let $\Gami$ be the plane graph obtained from $\Gi$ after adding for each $j<i$:
\begin{itemize}
 \item the edge $u_jv_j$ between the hubs of $L_j$, drawn in the face $F_j$, if $u_j$ and $v_j$ were not adjacent in $G_j$;
 \item two pendant edges $u_ja_j$ and $v_jb_j$, drawn in the face $F_j$, and such that $a_j, b_j$ are new vertices of degree $1$ in $\Gami$.
\end{itemize}
See \Cref{fig: Gamip} for an illustration.
Note that the fact that all $\Gami$ are indeed planar follows from the property that all interiors $F_j$ are pairwise disjoint.
Note also that for each $j\leq i$, $\Gamma_j[V(\Gamma_i)]$ is a plane subgraph of $\Gamma_i$.

The following claim is immediate and follows from the definition of $(\Gami)_i$.
\begin{claim}
 \label{clm: iiipropGamip}
 For each $0\leq i\leq \ell$, every edge $e$ from $E(\Gami)\setminus E(\Gi)$ has either one incident vertex of degree $1$, while the other vertex $u$ is a vertex of $G$ which is a hub of a lantern $L_j$ with $j<i$, or $e$ connects two hubs $u_j, v_j$ of $L_j$ and is drawn in $F_j$ for some $j<i$. 
\end{claim}

For each $0\leq i < \ell$, we partition the set of vertices of $G$ distinct from $u_i$ and $v_i$ and drawn in $\overline{F}_i$ into three sets $U_i, V_i, W_i$. We set $U_i:=U_{L_i}$, $V_i:=V_{L_i}$ and let $W_i$ denote the set of vertices drawn in $\overline{F}_i$ which are adjacent to both $u_i$ and $v_i$ (we reuse here the notations $U_L, V_L$ introduced in \Cref{sec:empty_lanterns}). Note that as $F_i$ is the interior of $L_i$, every vertex from $V(G)\setminus \sg{u_i, v_i}$ drawn in $\overline{F}_i$ belongs to exactly one of the three sets $U_i, V_i, W_i$.

We now define the sets $(R_i)_{0\leq i \leq \ell}$ of edges as follows. We start setting $R_0:=\emptyset$. For every $i\geq 1$, and every $j<i$, we add in $\Ri$ the following edges:
\begin{itemize}
 \item the edge $u_jv_j$;
 \item the edge $u_ja_j$ if $U_j\neq\emptyset$;
 \item the edge $v_jb_j$ if $V_j\neq\emptyset$.
\end{itemize}
Note that for every $j<i$, we have $R_j\subseteq \Ri$.
We say that an edge $e\in \Ri$ has \emph{type $j$} for $j\in \sg{1,2}$ if it has exactly $j$ endvertices belonging to $V(G)$. The construction is illustrated in \Cref{fig: Gamip}.

\begin{figure}[htb]
  \centering  
  \includegraphics[scale=1]{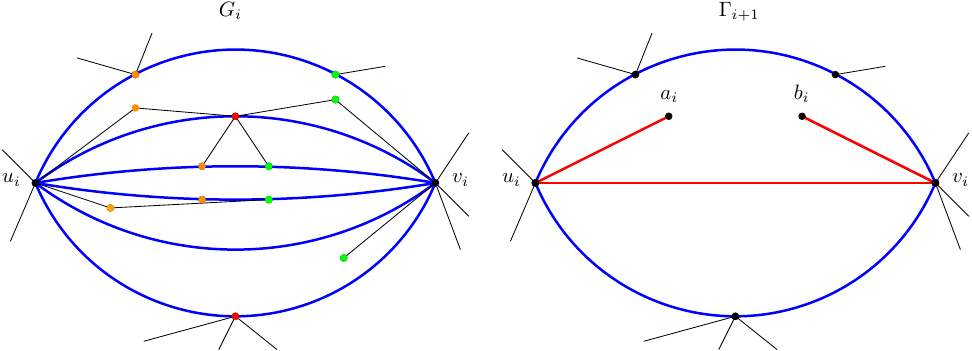}
  \caption{The construction of $\Gamip$. Left: in blue, a  nice lantern $L_i$. Vertices of $U_i$ (resp. $V_i$) are represented in orange (resp. green), and vertices of $W_i$ are represented in red.
  The free face of $L_i$ is the external face. Right: the graph $\Gamip$ obtained after emptying $L_i$. In red, we represented the edges from $\Rip$ added at step $i+1$. As $u_i$ and $v_i$ both have private neighbours in the interior of $L_i$ with respect to each other, both edges $u_ia_i$ and $v_ib_i$ are in $\Rip$.}
\label{fig: Gamip}
\end{figure}

\begin{claim}
 \label{clm: vNeighbouring}
 For each $0\leq i\leq \ell$, the set $\Ri$ is a neighbouring set of edges of $\Gami$.
\end{claim}

\begin{proof}
 We prove the result by induction on $i\in \sg{0,\ldots,\ell}$. For $i=0$, the result is immediate. We now let $0\leq i<\ell$, and assume that by induction hypothesis, for every $j\leq i$, $R_j$ is a neighbouring set of edges in $\Gamma_j$. We let $e,e'$ be two distinct edges from $\Rip$ and we will show that $e$ and $e'$ are at distance at most $1$ from each other in $\Gamip$. Note that, as $\Gamma_j[V(\Gamip)]$ is a subgraph of $\Gamip$ for every $j\leq i$, if $e,e'\in R_j$ for some $j\leq i$, then the induction hypothesis also implies that $e$ and $e'$ are at distance at most $1$ in $\Gamip$. 
 Note that by definition of $\Rip$, the result is also immediate when both $e$ and $e'$ belong to $E(\Gamip)\setminus E(\Gami)\subseteq \sg{u_iv_i, u_ia_i, v_ib_i}$.

 Hence, we may assume without loss of generality that $e\in E(\Gamip)\setminus E(\Gami)$, and that $e'\in E(\Gamip)\cap E(\Gami)$, and that $e$ and $e'$ do not share a common endvertex. Let $j\leq i$ be such that $e'$ was drawn in the face $F_j$. 
 Note that as $u_iv_i\in E(\Gamip)$, and as $e$ has at least one endvertex in $\sg{u_i, v_i}$, we may assume that $e'$ is neither incident to $u_i$, nor to $v_i$ in $\Gamip$.
 For the same reason, we may also assume that $e'$ is not at distance $1$ from both $u_i$ and $v_i$, so as $L_i$ is a short lantern, it implies that $e'$ is not an edge of $L_i$.
 Hence as $L_i$ is chordless, $e'$ must have at least one endvertex outside $\overline{F_i}$. We let $C$ and $C'$ denote respectively the boundary cycles of $F_i$ and $F_j$ in $\Gip$ (note that as $F_i$ and $F_j$ are disjoint, such facial cycles must exist in $\Gip$). 
 We now distinguish the following different cases:
 \begin{itemize}
  \item[$(1)$] Assume that both $e$ and $e'$ have type $2$, i.e. $e=u_iv_i$ and $e'=u_jv_j$; the situation is then as depicted in \Cref{fig: rouges22}. 
  As both $L_i$ and $L_j$ have width at least $6$, there exists some vertex $w$ (resp. $w'$) that belongs to $L_i-\sg{u_i, v_i}$ (resp. $L_j-\sg{u_j,v_j}$) whose only neighbours in $C$ (resp. in $C'$) are in $\sg{u_i, v_i}$ (resp. in $\sg{u_j, v_j}$). 
  As both $C$ and $C'$ separate $w$ from $w'$ in $G$, note that every path in $G\setminus \sg{u_j,v_j, u_i, v_i}$ connecting $w$ to $w'$ must have length at least $4$. 
  As $G$ has diameter $3$, it implies that there must be an edge $e''$ between the two sets $\sg{u_i, v_i}$ and $\sg{u, v}$ in $G$. By \Cref{clm: ipropGip}, $\Gip$ is an induced subgraph of $G$ so $e''\in E(\Gip)$, and by construction, $\Gip$ is a subgraph of $\Gamip$, so we also have $e''\in E(\Gamip)$, implying that $e$ and $e'$ are at distance $1$ apart in $\Gamip$. 

 \begin{figure}[htb]
  \centering  
  \includegraphics[scale=1.25]{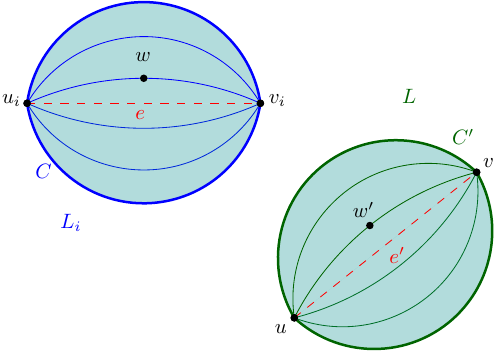}
  \caption{In blue, the lantern $L_i$ and in darkgreen the lantern $L$ when both $e$ and $e'$ have type $2$. Note that $C$ and $C'$ might not necessarily be disjoint.}
\label{fig: rouges22}
\end{figure}
  
 \item[$(2)$] Assume that $e$ has type $1$ and that $e'$ has type $2$. Then $e'=u_jv_j$, and without loss of generality, we may assume that $e=u_ia_i$. Moreover, as $u_jv_j\in E(\Gamip)$, we may also assume that $\sg{u_i, v_i}\cap \sg{u_j,v_j}=\emptyset$. By definition of $\Rip$, it implies that $U_i\neq \emptyset$. As $L_i$ is nice, it is a deep lantern of $G$, so there exists some vertex $w\in U_i$ such that for every $x\in V(C)\setminus \sg{v_i}$ we have $d_G(w, x)=d_{C+u_iw}(w, x)$. As $L_j$ has width at least $6$, there exists some vertex $w'$ in $G$ belonging to $L_j-\sg{u_j,v_j}$ whose neighbours in $C'$ are in $\sg{u_j,v_j}$ (in $G$). We now show that $u_i$ is at distance at most $1$ from $\sg{u_j,v_j}$ in $G$. Note that it is immediate if there exists some path of length $3$ in $G$ connecting $w$ to $w'$ and going through $u_i$. Note also that as both $C$ and $C'$ separate $w$ from $w'$, every path in $G - \sg{u_j, v_j, u_i, v_i}$ connecting $w$ to $w'$ has length at least $4$, hence as $G$ has diameter $3$, we may assume that there exists in $G$ some path $P$ of length $3$ from $w$ to $w'$ avoiding $u_i$. As $v_i\notin \sg{u_j, v_j}$, $v_i$ is at distance at least $2$ from $w'$ in $G$, and as $w\in U_i$, $v_i$ is also at distance at least $2$ from $w$ in $G$, so $P$ does not contain $v_i$. Then the only remaining option is that one of the vertices $u_j$ or $v_j$ belongs to $C\setminus \sg{u_i, v_i}$, and is at distance at most $2$ from $w$ in $G$ (see \Cref{fig: rouges12}). Assume without loss of generality that $u_j$ is such a vertex, and that $d_G(w,u_j)=2$. Then our choice of $w$ implies that $d_{C+u_iw}(w,u_j)=2$, hence $u_i$ and $u_j$ must be adjacent in $G$. In particular, as in case $(1)$, it allows us to conclude in this case that $e$ and $e'$ are also at distance $1$ apart in $\Gamip$.
   
\begin{figure}[htb]
  \centering    
  \includegraphics[scale=1.25]{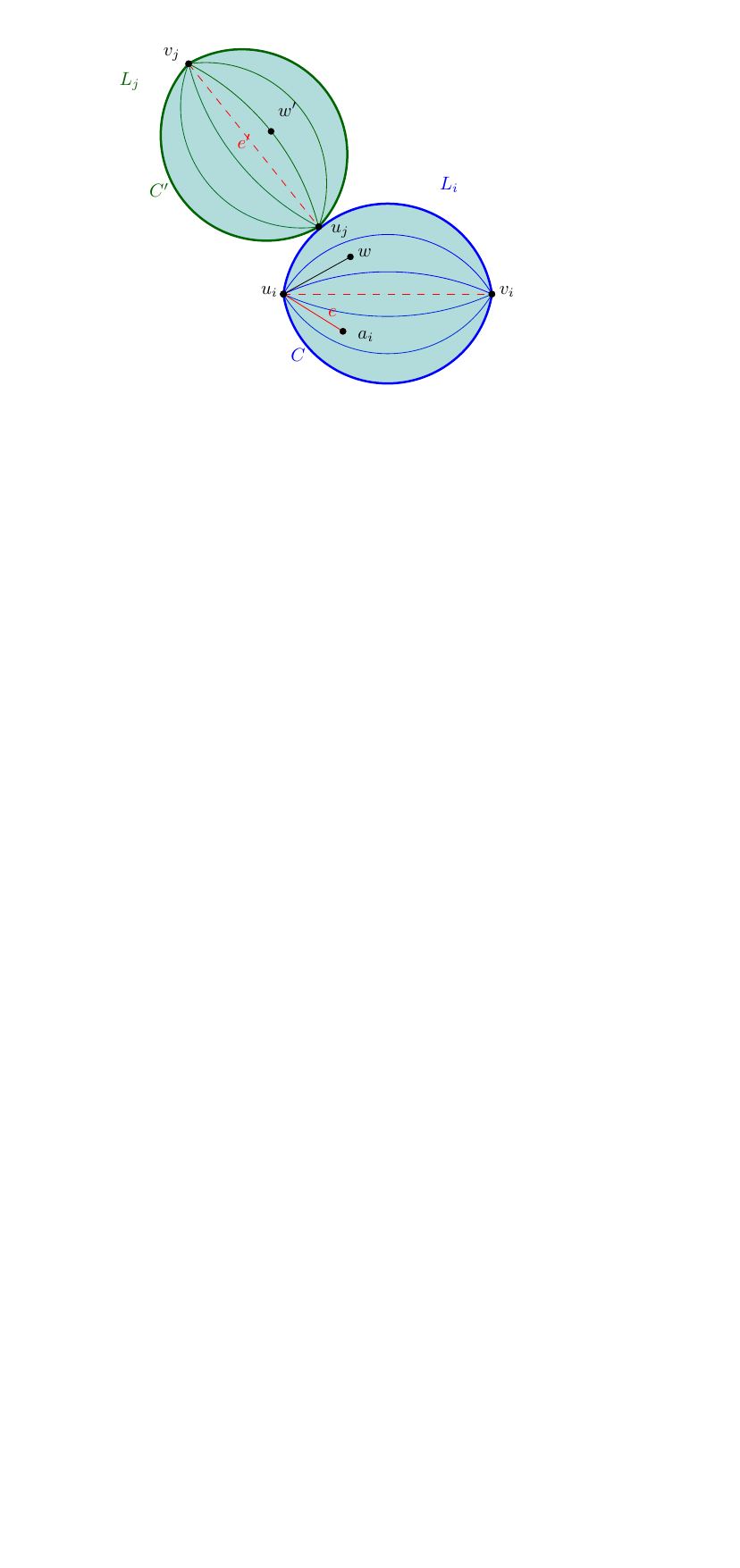}
  \caption{Configuration in the end of the proof of case $(2)$.}
\label{fig: rouges12}  
\end{figure} 

 \item[$(3)$] We claim that the case where $e$ has type $2$ and $e'$ has type $1$ is symmetric to case $(2)$. We eventually assume that both $e$ and $e'$ have type $1$. Without loss of generality, we assume that $e=u_ia_i$, and that $e'=u_ja_j$. We moreover assume that $u_j\neq u_i$, $u_i\neq v_j$ and $u_j\neq v_i$, as otherwise we conclude immediately.
 The arguments from case $(2)$ apply identically, and imply that there exist some vertices $w,w'$ respectively in $U_i$ and $U_j$ such that for every $x\in V(C)\setminus \sg{v_i}$ (resp. $x\in V(C')\setminus \sg{v_j}$) we have $d_G(w, x)=d_{C+u_iw}(w, x)$ (resp. $d_G(w', x)=d_{C'+u_jw'}(w', x)$). We now show that $u_i$ and $u_j$ are at distance at most $1$ in $G$. First, note that as $C$ and $C'$ separate $w$ from $w'$ in $G$, and as $u_i$ and $u_j$ are respectively the only neighbours of $w$ and $w'$ on $C$ and $C'$, every path in $G$ connecting $w$ to $w'$ in $G-\sg{u_i, u_j}$ has length at least $4$. We consider a path $P$ in $G$ of length at most $3$ from $w$ to $w'$ (which exists as $G$ has diameter at most $3$). Note that if both $u_j$ and $u_i$ belong to $P$, then $u_j$ and $u_i$ are at distance at most $1$ in $G$, so we assume that it is not the case. Assume also without loss of generality that $u_i$ is adjacent to $w$ on $P$, the other case ($u$ adjacent to $w'$ on $P$) being symmetric. Then $u_i$ must be at distance $2$ from $w'$ in $G$, and our choice of $w'$ implies in particular that $u_i\in V(C')$ and that $d_{C'+u_jw'}(u_i, w')=2$. In particular, it implies that $u_i$ and $u_j$ must be adjacent in $G$. To conclude, we argue that for the same reasons than in case $(1)$, we must also have $u_iu_j\in E(\Gamip)$, implying that $e$ and $e'$ are at distance $1$ apart in $\Gamip$.
\end{itemize}
 In each case, we were able to show that $e$ and $e'$ are at distance at most $1$ from each other in $\Gamip$, implying that $\Rip$ is indeed a neighbouring set in $\Gamip$.
\end{proof}

\paragraph*{Construction of $\mathbf{(\mu_i)_i}$}
We reuse here all notations from previous paragraphs. We will now define for each $0\leq i\leq \ell$ a fractional matching $\mu_i$ of the graph $\Gami[\Ri]$ that will intuitively allow us to count the number of vertices in $V(G)\setminus V(G_i)$. 
We set, for each $0\leq i\leq \ell$, $X_i:=V(G)\setminus V(G_i)$. Note that the set $X_i$ corresponds exactly to the set of vertices of $G$ drawn in the interiors $F_j$, for $j<i$, and that every vertex from $\Xip \setminus X_i$ is drawn in $F_i$.

\begin{claim}
 \label{clm: viPropfi}
 For each $i\in \sg{0,\ldots,\ell}$, there exists a mapping $f_i: X_i\to \Ri$ such that for each $v\in X_i$, if $f_{i}(v)=xy$, then for every vertex $w\in \sg{x,y}\cap V(G)$, 
  $v\in N_G(w)$.
\end{claim}

\begin{proof}
 We prove the claim by induction on $i\geq 0$. As $X_0=\emptyset$, the result is immediate when $i=0$. Assume now that $0\leq i<\ell$, and that $f_i: X_i\to \Ri$ was already constructed with the desired properties.
 We define $\fip: \Xip \to \Rip$ by setting for every vertex $x\in \Xip$:
 
$$\fip(x)=
\begin{cases}
u_iv_i \text{ if }x\in W_i,\\ 
u_ia_i \text{ if }x\in U_i,\\
v_ib_i \text{ if }x\in V_i,\\
f_i(x) \text{ otherwise}.
\end{cases}$$

Recall that as $L_i$ is a dominating lantern in $\Gi$, every vertex of $\Gi$ which is drawn in $F_i$ is in $W_i\cup U_i \cup V_i$. In particular, as we already observed that vertices of $\Xip\setminus X_i$ are drawn in $F_i$, it implies that $\Xip \setminus (W_i\cup U_i \cup V_i)\subseteq X_i$, hence $\fip$ is well-defined. Moreover, recall that $u_ia_i$ (resp. $v_ib_i$) belongs to $\Rip$ if and only if $U_i\neq \emptyset$ (resp. $V_i\neq \emptyset$), so $\fip$ takes values in $\Rip$.

Now, observe that the definition of $\fip$ together with the induction hypothesis immediately conclude the proof.
\end{proof}

\begin{claim}
 \label{clm: viPropmui}
 For each $i\in \sg{0,\ldots,\ell}$, there exists a fractional matching $\mu_i: \Ri\to [0,1]$ of $\Gami[\Ri]$, such that 
  $$\left|X_i\right|= \left(\sum_{e\in \Ri}\muip(e)\right)\Delta.$$
\end{claim}
\begin{proof}
 We let $(f_i)_{0\leq i<\ell}$ be the sequence of matchings given by \Cref{clm: viPropfi}, and define for each $i\in \sg{0, \ldots, \ell}$ the mapping $\mui: \Ri \to \mathbb{R}^+$ by setting for each $e\in \Ri$ 

$$ \mui(e):=\frac{|f^{-1}_{i}(e)|}{\Delta}.$$
 
 First, we check that $\mui(e)\in [0,1]$ for each $e\in \Ri$. By definition of $\Ri$, $e$ has at least one endvertex $v$ in $V(G)$. In particular, by \Cref{clm: viPropfi}, every vertex $x\in X_i$ such that $f_i(x)=e$ must be adjacent to $v$, hence $f_i^{-1}(e)\subseteq N_G(v)$. Therefore, $|f_i^{-1}(e)|\leq \Delta$, implying that $\mui(e)\in [0,1]$.

In fact, the same arguments also imply that $\mui$ is a fractional matching: let $v\in V(G)$. Again, \Cref{clm: viPropfi} implies that for every edge $e\in \Ri$ which is incident to $v$, we have $f_i^{-1}(e)\subseteq N_G(v)$. Thus in particular, if $e_1, \ldots, e_k$ denote the edges from $\Ri$ incident to $v$, we have

$$\sum_{j=1}^k \mui(e_j) = \sum_{j=1}^k \frac{|f_i^{-1}(e_j)|}{\Delta}\leq \frac{|N_G(v)|}{\Delta}\leq 1.$$

This implies that $\mui$ is indeed a fractional matching. 
To conclude the proof of the claim, observe that

$$\left|X_i\right|
= \sum_{e\in \Ri}|f_i^{-1}(e)|
= \left(\sum_{e\in \Ri}\mui(e)\right)\Delta.$$

\end{proof}

\subsection{Proof of \texorpdfstring{\Cref{thm: main1}}{Theorem~\ref{thm: main1}}}
\label{sec: vidage-2}

\paragraph{Proof overview}
We now consider the sequence $(\Gi, \Gami, \Ri, \mui)_{0 \leq i\leq \ell}$ constructed previously. In particular, by \Cref{clm: ipropGip}, the graph $\Gell$ does not admit any $39$-lantern. As before, we let $\Xell:=V(G)\setminus V(\Gell)$. We will now show that we can find a set $X\subseteq V(\Gell)$ such that $|\Xell\cup X|\leq \tfrac92 \Delta$, and that $|V(\Gell)\setminus X|\leq 9 + 39^3$. Intuitively, we will show that for some vertices $u$ of $\Gell$ with large degree, we can attach a new pendant edge $e$ to $u$ in $\Gamell$, so that $e$ is at distance at most $1$ from every other edge of $\Rell$. We let $R$ denote the set of such pendant edges we added, and $\Gamma$ be the graph obtained from $\Gamell$ after adding every pendant edge from $R$. We will show that $R$ is still a neighbouring set in $\Gamma$, and extend our previous fractional matching $\muell$ to some fractional matching $\mu$ defined on $R\cup \Rell$. We will then show that up to $9+39^3$ vertices, we can map every vertex $v$ of $V(\Gell)$ to some edge $e$ of $R\cup \Rell$ so that $v$ is adjacent to every endvertex of $e$ that belongs to $V(G)$. $X$ will be defined as the set of vertices $V(\Gell)$ which can be mapped in such a way, and we will prove that we have $|\Xell\cup X|\leq \mu(\Gamma, \Rell\uplus R)\cdot \Delta$. The fact that $|V(\Gell)\setminus X|\leq 9+39^3$ relies on the fact that every planar graph of diameter $3$ has a dominating set of constant size \cite{MS96, DGH06}.

\paragraph{Preliminary results}
We reuse in this section all notations from \Cref{sec: vidage-1}.
Before defining the sets $\Gamma$ and $R$, we need a few preliminary results, describing the structure of $\Gell$. Recall that by \Cref{clm: ipropGip}, the graph $\Gell$ does not contain a $39$-lantern. In particular, the following remark immediately follows.

\begin{remark}
 \label{rem: common-neighbours}
 For every two vertices $u,v\in V(\Gell)$, note that we have $|N_{\Gell}(u)\cap N_{\Gell}(v)|< 39$.
\end{remark}
 
The next claim follows from the proof of \cite[Lemma 5.1]{GH02}. 

\begin{claim}[Lemma 5.1 in \cite{GH02}]
 \label{clm: ecc3}
 Every vertex of $\Gell$ of eccentricity $3$ has degree strictly less than $39^3$.
\end{claim}

\begin{claim}
\label{clm: dist2-neighbours}
 For every vertex $v\in V(\Gell)$, if $u\in V(\Gell)$ is such that $d_{\Gell}(u,v)= 2$, then the set $\sg{w\in N_{\Gell}(v): d_{\Gell}(w,u)\leq 2}$ has size at most $39^2$.
\end{claim}

\begin{proof} 
 We set $S:=\sg{w\in N_{\Gell}(v): d_{\Gell}(w,u)\leq 2}$ and $k:=|S|$.
 We partition $S$ into two sets by setting $S_1:=N_{\Gell}(u)\cap N_{\Gell}(v)$ and $S_2:= S\setminus S_1$. 
 By \Cref{rem: common-neighbours}, $|S_1|\leq 38$. 
 
 To bound the size of $S_2$, construct a BFS-tree $T$ of depth $2$ rooted in $u$ whose leaves correspond exactly to elements of $S_2$. Note that it is always possible to find such a tree because every shortest path from $u$ to any vertex of $S_2$ avoids both $v$ and every other vertex from $S_2$.
 Again, \Cref{rem: common-neighbours} implies that the internal nodes of $T$ distinct from $u$ are each adjacent to at most $38$ leaves in $T$. Hence $T$ has at least 
 $\frac{|S_2|}{38}$ internal nodes distinct from $u$. As $\Gell$ has no $39$-lantern, $T$ contains at most $38$ edge-disjoint paths from $u$ to the elements of $S_2$, implying that $\frac{|S_2|}{38}\leq 38$.
 
 We thus obtain 
 $$k\leq 38 + 38^2 \leq 39^2.$$
\end{proof}

\begin{claim}
 \label{clm: close-to-hubs}
 Let $L$ be a nice lantern in $G$ with hubs $u,v$, and let $w\in V(G)\setminus \sg{u,v}$ be a vertex drawn in the closure of the free face of $L$. Then $w$ is at distance at most $2$ from the set $\sg{u,v}$ in $G$. Moreover, if $d_G(w,v)\geq 2$ and $U_L\neq \emptyset$ (resp. if $d_G(w,u)\geq 2$ and $V_L\neq \emptyset$), then $d_G(u,w)\leq 2$ (resp. $d_G(v,w)\leq 2$).
\end{claim}

\begin{proof}
 Let $F$ denote the interior of $L$, and $C$ its boundary cycle.
 First, note that as $L$ is a short lantern, if $w$ belongs to $C$ then the result is immediate, thus we may assume that $w$ belongs to the free face of $L$.
 
 Assume first for a contradiction that $w$ is at distance $3$ from both $u$ and $v$. 
 As $L$ has width at least $6$, there exists a vertex $w'$ drawn in $F$ such that no vertex of $C\setminus \sg{u,v}$ is adjacent to $w'$. As $G$ has diameter $3$, we then obtain a contradiction as every shortest path from $w$ to $w'$ must intersect $C$.
 
 Assume now without loss of generality that $U_L\neq \emptyset$, and that $d_G(w, v)\geq 2$, the other case being symmetric. As $L$ is deep, there exists some vertex $u'\in U_L$ such that for every $x\in V(C)\setminus \sg{u,v}$, we have $d_G(u',x)=d_{C+uu'}(u',x)$. As $G$ has diameter $3$, there is a path $P$ of length at most $3$ from $w$ to $x$. As $C$ separates $w$ from $u'$, $P$ intersects $C$. Note that if $u$ belongs to $P$, then it immediately implies that $d_G(u,w)\leq 2$, so we assume without loss of generality that $P$ avoids $u$.
 As $u'\in U_L$, $d_G(u',v)\geq 2$, so $P$ cannot contain $v$, and it must intersect $C - \sg{u,v}$. Note that our choice of $u'$ implies that $u'$ is at distance at least $2$ in $G$ from every vertex of $C - \sg{u,v}$, hence if we let $x$ denote the neighbour of $w$ on $P$, $x$ must belong to $C-\sg{u,v}$. In particular, if $x\in N_G(u)$, then we get $d_G(w, u)\leq 2$. On the other hand, if $x$ is not a neighbour of $u$, then we must have $d_G(u', x)=d_{C+uu'}(u',x)=3$, which contradicts our assumption that $P$ has length at most $3$.
\end{proof}

\begin{claim}
 \label{clm: ecc2-big-degree}
 For every vertex $w\in V(\Gell)$ of eccentricity $2$ in $\Gell$, if $\deg_{\Gell}(w)> 2\times 39^2$, then $w$ is at distance at most $1$ from every edge from $\Rell$ in $\Gamell$.
\end{claim}

\begin{proof}
 Let $w$ be a vertex of eccentricity $2$ in $\Gell$ and let $e\in \Rell$. We set $k:=\deg_{\Gell}(w)$ and assume that $k>2\times 39^2$. As in \Cref{sec: vidage-1}, we say that $e$ has \emph{type $j$} for $j\in \sg{1,2}$ if it has exactly $j$ endvertices in $V(G)$. By definition of $\Rell$, every edge from $\Rell$ has either type $1$ or $2$.
 
 \begin{itemize}
  \item Assume first that $e$ has type $2$, and write $e=u_iv_i$, for some $i\in \sg{0, \ldots, \ell-1}$. As $w$ is a vertex of $\Gell$, it is not drawn in $F_{i}$. Assume for a contradiction that $w$ is at distance at least $2$ from both $u_i$ and $v_i$ in $\Gamell$, and thus also in $\Gell$. As $w$ has eccentricity $2$ in $\Gell$, it implies that $d_{\Gell}(u,w)=d_{\Gell}(v,w)=2$. In particular, as $L_{i}$ is dominating in $G$, $w$ does not belong to the cycle boundary of $F_i$ and  its neighbours in $\Gell$ also belong to the closure of the free face of $L_i$.
  Hence, by \Cref{clm: close-to-hubs} every neighbour of $w$ in $\Gell$ is at distance at most $2$ from at least one of the two vertices $u_i,v_i$. Assume without loss of generality that $\ceil{\frac{k}{2}}$ neighbours of $w$ are at distance at most $2$ from $u_i$. Then \Cref{clm: dist2-neighbours} implies that
  
  $$\ceil{\frac{k}{2}}\leq 39^2.$$
  
  In particular, we obtain a contradiction as we assumed that $k> 2\times39^2$.
  
  \item Assume now that $e$ has type $1$, and write $e=u_ia_i$ for some $i\in \sg{0,\ldots, \ell-1}$, with $u_i\in V(G)$ and $a_i\in V(\Gamell)\setminus V(G)$. In particular, $U_i\neq \emptyset$. First, observe that if $w=v_i$, then by definition of $\Gamell$, we must have $u_iv_i\in E(\Gamell)$, so $w$ is adjacent to the edge $e$ in $\Gamell$ and there is nothing to prove, hence we assume that $w\neq v_i$. Our goal now is to show that $d_{\Gamell}(u_i,w)\leq 1$. 
  We assume for a contradiction that it is not the case.
  As $w$ has eccentricity $2$ in $\Gell$, we then have $d_{\Gell}(u_i,w)=2$, and as $\Gell$ is an isometric subgraph of $G$, we also have $d_G(u_i,w)=2$.
  Recall that no vertex of $\Gell$ is drawn in $F_i$, so \Cref{clm: close-to-hubs} implies that every vertex of $N_{\Gell}(w)\setminus \sg{v_i}$ is either adjacent to $v_i$ or lies at distance at most $2$ from $u_i$ in $G$, and thus also in $\Gell$. By \Cref{rem: common-neighbours}, there are at most $38$ vertices from $N_{\Gell}(w)\setminus \sg{v_i}$ that are also neighbours of $v_i$, hence at least $k-39$ vertices from $N_{\Gell}(w)$ must be at distance at most $2$ from $u_i$ in $\Gell$. In particular, \Cref{clm: dist2-neighbours} implies that
  
  $$k-39\leq 39^2.$$

  We then obtain a contradiction as we assumed that $k>2\times 39^2$.
 \end{itemize}
\end{proof}

\paragraph*{Construction of $\mathbf{(\Gamma, R)}$}
We now have everything we need to construct $(\Gamma, R)$. We let $W$ be the set of vertices $w\in V(\Gamell)\cap V(\Gell)$ 
satisfying one of the following conditions
\begin{itemize}
 \item $w$ is a vertex of eccentricity $1$ in $\Gell$,
 \item $w$ is a vertex of eccentricity $2$ in $\Gell$ and degree strictly greater than $2\times39^2$. 
\end{itemize}
We define the graph $\Gamma$ from $\Gamell$ after attaching a pendant edge $wc$ to every vertex $w\in W$, where $c$ denotes a new vertex distinct form the ones of $\Gamell$. We let $R:=E(\Gamma)\setminus E(\Gamell)$ denote the set of such pendant edges $wc$. 
Note that as $\Gamell$ is planar, $\Gamma$ is also planar. 

\begin{claim}
\label{clm: xPropR}
The set $\Rell\cup R$ is a neighbouring set of edges in $\Gamma$.
\end{claim}

\begin{proof}
 We let $e, e'$ be two distinct edges from $\Rell\uplus R$, and show that $e$ and $e'$ must lie at distance at most $1$ from each other in $\Gamell$. Note that it follows immediately from \Cref{clm: vNeighbouring} when $e,e'\in \Rell$, hence we may assume that $e\in R$ and write $e=wc$, with $w\in W$ and $c\in V(\Gamma)\setminus V(\Gamell)$.

Assume first that $e'\in \Rell$. Then, $e'$ has at least one endvertex in $V(\Gi)$. In particular, if $w$ has eccentricity $1$ in $\Gell$, $e$ and $e'$ are clearly at distance at most $1$ from each other in $\Gell$, and thus also in $\Gamma$.
Assume now that $w$ has eccentricity $2$ in $V(\Gell)$ and degree strictly greater than $2\times39^2$. Then \Cref{clm: ecc2-big-degree} implies that $w$ is at distance at most $1$ from $e'$ in $\Gamell$, so we are done when $e'\in \Rell$.

We now assume that $e'\in R$ and write $e'=w'c'$ with $w'\in W$ and $c'\in V(\Gamma)\setminus V(\Gamell)$. Again, if one of the two vertices $w$ and $w'$ has eccentricity $1$ in $\Gell$, then $e$ and $e'$ clearly lie at distance at most $1$ from each other in $\Gamma$. We now assume that both $w$ and $w'$ have eccentricity $2$ and degree strictly greater than $2\times39^2$. Then 
applying \Cref{clm: dist2-neighbours} with $w$ and $w'$ playing respectively the roles of $u$ and $v$ implies that $w$ and $w'$ must lie at distance at most $1$ from each other in $\Gell$, and thus also in $\Gamma$, so we are also done when $e'\in R$.
\end{proof}

\paragraph*{Construction of $\mathbf{\mu}$}
We now let $X$ be the set of vertices from $V(\Gell)\cap V(\Gamma)$ incident to some edge of $R$.

\begin{claim}
 \label{clm: xiPropf}
 There exists a mapping $f: \Xell\cup X\to \Rell\cup R$ such that for each $v\in \Xell\cup X$, if $f(v)=xy$, then for every vertex $w\in \sg{x,y}\cap V(G)$, $v\in N_G(w)$.
\end{claim}
\begin{proof}
 We define a mapping $f: \Xell\cup X\to \Ri\cup R$ as follows. We consider $f_{\ell}: \Xell \to \Rell$ be the mapping given by \Cref{clm: viPropfi},
 and for each $x\in \Xell$, we set $f(x):=f_{\ell}(x)$. Moreover, for each $x\in X$, we let $f(x)$ be any edge $e\in R$ such that $e=wc$ with $w\in W\cap N_{\Gell}(x)$ and $c\in V(\Gamma)\setminus V(\Gamell)$ (take an arbitrary such edge if there are multiple candidates for $e$). 
 
 The fact that $f$ satisfies the desired property follows immediatly from its definition, and from \Cref{clm: viPropfi}.
\end{proof}

\begin{claim}
 \label{clm: xiiPropmu}
 There exists a fractional matching $\mu: \Rell \cup R \to [0,1]$ of $\Gamma[\Rell\cup R]$ such that 
 $$|\Xell \cup X|\leq \left(\sum_{e\in \Rell\cup R}\mu(e)\right)\Delta$$
 and
 $$|V(\Gell)\setminus X|\leq 9+39^3.$$
\end{claim}

\begin{proof}
 We define $\mu: R\cup \Rell\to \mathbb R^+$ by setting for each edge $e\in R\cup \Ri$ 
$$\mu(e):=\frac{|f^{-1}(e)|}{\Delta}.$$
 The proof that $\mu$ has values in $[0,1]$ and that it is a fractional matching of $\Gamma[\Rell \cup R]$ works exactly the same way than for the one of \Cref{clm: viPropmui}:
 observe that 
 for each $e\in R\cup \Rell$ and every endvertex $w$ of $e$ such that $w\in V(G)$, \Cref{clm: xiPropf} implies that $f^{-1}(e)\subseteq N_G(w)$. In particular, it gives $\mu(e)\leq 1$, and also implies that $\mu$ is a fractional matching of $\Gamma$.
 Moreover, the definition of $\mu$ together with \Cref{clm: xiPropf} and the fact that $X$ and $\Xell$ is disjoint imply that we have
 
\begin{align*}
|\Xell\cup X| & = \left(\sum_{e\in \Rell}\muell(e)\right)\Delta + \sum_{e\in R}|f^{-1}(e)| & \\
& = \left(\sum_{e\in \Rell}\muell(e)\right)\Delta + \left(\sum_{e\in R}\mu(e)\right)\Delta & \\
& =\left(\sum_{e\in \Rell\uplus R}\mu(e)\right)\Delta.
\end{align*}

We now prove the last inequality of the claim. For this we will need to use that planar graphs of diameter at most $3$ admits dominating sets of constant size. More precisely, improving a result of MacGillivray and Seyffarth \cite{MS96}, Dorfling, Goddard and Henning \cite[Theorem 3]{DGH06} proved that every planar graph of diameter at most $3$ has domination number at most $9$. In particular, using \Cref{clm: ipropGip} this implies that there exists a set $D\subseteq V(\Gell)$ such that $V(\Gell)= N_{\Gell}[D]$ and $|D|\leq 9$.

\begin{claim}
 \label{clm: small-degree-outside-D}
 Every vertex of $V(\Gell)\setminus D$ has degree at most $9(39+1) - 1$ in $\Gell$. 
\end{claim}

\begin{proof}
 If $u\in V(\Gell)\setminus D$ is such that $\deg_{\Gell}(u)\geq 9(39+1)$, then by pigeonhole principle, there exists $v\in D$ such that $u$ and $v$ have at least $39$ common neighbours, contradicting \Cref{rem: common-neighbours}.
\end{proof}

\begin{claim}
 \label{clm: bd-size}
 We have
 $$|V(\Gell)\setminus N_{\Gell}(W\cap D)|\leq 9+39^3.$$
\end{claim}

\begin{proof}
 We set $S:=V(\Gell)\setminus N_{\Gell}(W\cap D)$. As $D$ is a dominating set of $\Gell$, note that every vertex of $S$ which does not belong to $D$ must admit a neighbour in $D\setminus W$. Let $w\in D\setminus W$. Then by definition of $W$, either $w$ has eccentricity $2$ and degree at most $2\times 39^2$ in $\Gell$, or it has eccentricity $3$ in $\Gell$ and thus degree at most $39^3$ in $G_i$ by \Cref{clm: ecc3}. Hence, $w$ has degree at most $39^3$ in $\Gell$. As $D\setminus W$ dominates $S$, it implies that 
 $$|S|= |S\cap D| + |S\setminus D|  \leq 9 + 39^3.$$
\end{proof}

To conclude the proof of $(xiii)$, we just need to combine \Cref{clm: bd-size} with the observation that the following inclusion holds
$$V(\Gell)\setminus X \subseteq V(\Gell) \setminus N_{\Gell}(W\cap D)).$$ 
\end{proof}

\paragraph*{Proof of \Cref{thm: main1}}
The proof of \Cref{thm: main1} now follows from \Cref{clm: xPropR,clm: xiiPropmu}: we consider $\Gamma, R, \Rell, \mu$ constructed above. By \Cref{clm: xPropR}, $R\cup \Rell$ is a neighbouring set of edges of $\Gamma$. Moreover, by \Cref{clm: xiiPropmu}, $\mu$ is a fractional matching of $\Gamma[R\cup \Rell]$ and we have:

\begin{align*}
|V(G)| & = |X\cup \Xell| + |V(\Gell)\setminus X| \leq \left(\sum_{e\in \Rell\uplus R}\mu(\Gamma[R])\right)\Delta + 9 + 39^3.
\end{align*}
\qed

\section{Fractional matching number of planar neighbouring sets of planar graphs.}
\label{sec: fractional-matching}

In this section, we show that for every planar graph $G$ and every neighbouring subset of edges $R$, we have $\mu^{\ast}(G[R])\leq \tfrac92$ (see \Cref{thm: frac-intro}). In fact, we prove that this inequality holds if we assume that $G$ only excludes the complete graph $K_5$ as a minor, instead of assuming that $G$ is planar.
In particular, if we combine it with \Cref{thm: main1}, it immediately implies that \Cref{thm: main} holds.

\subsection{Excluded subgraphs.}
\label{sec: obstructions}
We prove in this subsection that if $R$ is a neighbouring subset of edges of a $K_5$-minor-free graph $G$, then $G[R]$ must exclude as subgraphs the graphs $H_1:=2K_3 + 2K_2, H_2:=C_5+2K_2, H_3:=C_7+K_3$ and $H_4:= C_9$.
We depicted $H_1, H_2$ and $H_3$ in \Cref{fig: obstructions}. 

\begin{figure}[htb]
  \centering   
  \includegraphics[scale=0.8]{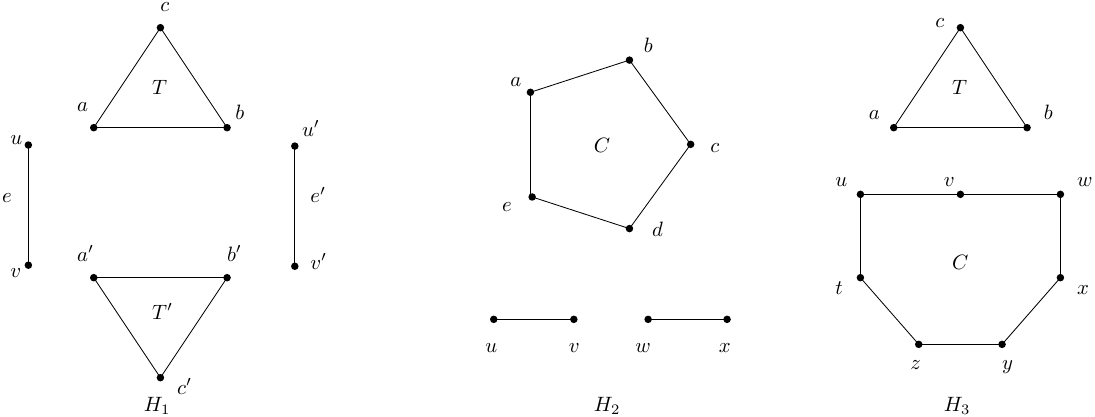}
  \caption{Three forbidden subgraphs of $G[R]$ when $G$ is $K_5$-minor-free and $R$ is a neighbouring set of edges.}
  \label{fig: obstructions}
\end{figure} 
 
\subsubsection{Excluding \texorpdfstring{$H_1, H_2$ or $H_3$}.} 

\begin{lemma}
 \label{lem: H1}
 If $G$ is a $K_5$-minor-free graph and $R$ is a neighbouring set of edges of $G$, then $G[R]$ does not contain $H_1$ as a subgraph.
\end{lemma}

\begin{proof}
 We assume for a contradiction that $G[R]$ contains $H_1$ as a subgraph, and reuse all the variable names from \Cref{fig: obstructions}. 
 We set $V_1:=\sg{u,v}$ and $V_2:=\sg{u',v'}$. Note that since $R$ is a neighbouring set, for every edge $f=xy\in R$ and every triangle $D$ of $G[R]$ vertex-disjoint from $f$, there are at least two vertices of $D$ adjacent to $\sg{x,y}$. In particular, both $V_1$ and $V_2$ are adjacent to at least two vertices of each of the triangles $T,T'$.
 We distinguish two cases. 

\mbox{\bf{Case 1}:} A vertex of $V(T)$, say $c$, is not adjacent to any vertex of $T'$ in $G$. 
Again, we distinguish two subcases.
  \begin{itemize}
   \item[$\bullet$] Assume first that both $V_1$ and $V_2$ are connected to $c$ by an edge, and that $V_1$ is adjacent to $a$ while $V_2$ is adjacent to $b$ (see \Cref{fig: H1-Case1}, left).
   Since $c$ is not adjacent to any vertex of $T'$, $bc\in R$, and $R$ is a neighbouring set of edges, it holds that $b$ is adjacent to at least two vertices from $T'$. In particular, by the pigeonhole principle, $V_1$ and $b$ have a common neighbour on $T'$, and we assume without loss of generality that it is $a'$. Let $V_3=\{a,c\}$, $V_4=\{b',c'\}$ and $V_5=\{a',b\}$. Observe that $V_1$, $V_2$, $V_3$ and $V_4$ are all edges of the neighbouring set and thus are at distance~1 from each other, so contracting them results in pairwise adjacent vertices, and by construction $V_5$ is at distance~1 from every other $V_i$.
   Therefore by contracting all the $V_i$'s, we obtain a $K_5$ as minor, a contradiction. 
   
   \item[$\bullet$] We assume now that we are not in the previous subcase, and that by symmetry, we are also not in the situation where $V_1$ is adjacent to $c$ and $b$ and $V_2$ is adjacent to $c$ and $a$. In particular, since $V_1$ and $V_2$ both have to be adjacent to two vertices of $T$, both are adjacent to a common vertex of $T-c$; assume without loss of generality that $b$ is such a vertex (see \Cref{fig: H1-Case1}, right). 
   Since $c$ is not adjacent to any vertex of $T'$, $ac,bc\in R$, and $R$ is a neighbouring set, it holds that $b$ and $a$ are adjacent to $T'$. 
   Let $V_3:=\sg{a,c}, V_4:=\sg{b}$ and $V_5:=\sg{a',b',c'}$. The sets $V_1$, $V_2$, $V_3$ and $V_5$ are disjoint and at distance~1 from each other and $V_4$ is adjacent to every other $V_i$, so by contracting all the $V_i$'s, we obtain a $K_5$ as minor, a contradiction. 
   
   \begin{figure}[htb]
  \centering   
  \includegraphics[scale=1]{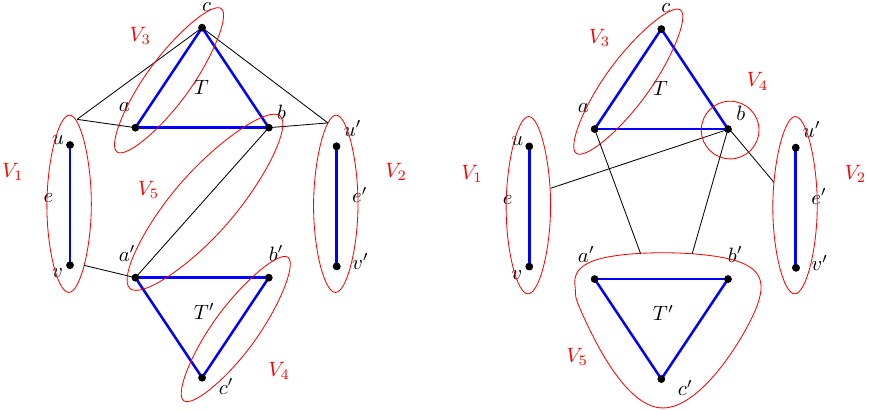}
  \caption{The two distinct situations in Case 1 of the proof of \Cref{lem: H1}. The edges of $R$ are represented in blue, while the black edges denote the edges from $G$ (which do not necessarily belong to $R$).} 
  \label{fig: H1-Case1}
\end{figure} 
  \end{itemize}

 \mbox{\bf{Case 2}:} Every vertex $x\in V(T)$ is adjacent to some vertex of $T'$ in $G$. As both $V_1$ and $V_2$ are adjacent to at least two vertices of $T$, we may assume without loss of generality that both $V_1$ and $V_2$ are adjacent to $b$. By hypothesis, $b$ is adjacent to at least one vertex of $T'$. We can thus find a $K_5$ minor the exact same way that we did for the second subcase of Case 1.  
\end{proof}

\begin{lemma}
 \label{lem: H2}
 If $G$ is a $K_5$-minor-free graph and $R$ is a neighbouring set of edges of $G$, then $G[R]$ does not contain $H_2$ as a subgraph.
\end{lemma}

\begin{proof}
  We assume for a contradiction that $G[R]$ contains $H_2$ as a subgraph, and reuse all the variable names from \Cref{fig: obstructions}. 
  Let $V_1:=\sg{u,v}$ and $V_2:=\sg{w,x}$. As $R$ is a neighbouring set in $G$,  $V_1$ and $V_2$ must be adjacent and each of them must be adjacent to at least three vertices of the cycle $C$. 
  In particular, there is a vertex in $C$, without loss of generality $d$, adjacent to both $V_1$ and $V_2$.
  Let $V_3:=\sg{a,e}$ and $V_4:=\sg{b,c}$ and $V_5=\{d\}$ (see \Cref{fig: H2}). The sets $V_1$, $V_2$, $V_3$ and $V_4$ are disjoint and at distance~1 from each other and $V_5$ is adjacent to every $V_i$, so the graph obtained after contracting each $V_i$ and removing the other vertices from $G$ is a $K_5$, a contradiction. 
\end{proof}
  
     \begin{figure}[htb]
  \centering   
  \includegraphics[scale=1.25]{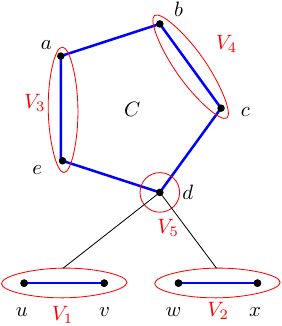}
  \caption{Proof of \Cref{lem: H2}. The edges of $R$ are represented in blue, while the black edges denote the edges from $G$.} 
  \label{fig: H2}
\end{figure}

\begin{lemma}
 \label{lem: H3}
 If $G$ is a $K_5$-minor-free graph and $R$ is a neighbouring set of edges of $G$, then $G[R]$ does not contain $H_3$ as a subgraph.
\end{lemma}

\begin{proof}
  We assume for a contradiction that $G[R]$ contains $H_3$ as a subgraph, and reuse all the variable names from \Cref{fig: obstructions}. 
  Let $V_1:=\sg{t,u}, V_2:=\sg{w,x}$ and $V_3:=\sg{y,z}$. As $R$ is a neighbouring set in $G$, each of the sets $V_1, V_2$ and $V_3$ must be adjacent to at least two vertices of the triangle $T$.

  \mbox{\bf{Case 1}:}
  Assume first that there exists some vertex of $T$, say $a$, which is adjacent to each $V_i$. Let $V_4:=\sg{b,c}$ and $V_5:=\sg{a}$. The minor obtained by contracting the $V_i$'s is a $K_5$, a contradiction.
  
  \mbox{\bf{Case 2}:}
  Assume now that no vertex of $T$ is simultaneously adjacent to 
  $V_1, V_2$ and $V_3$. Then we may assume without loss of generality that $N_T(V_1)=\sg{a,c}, N_T(V_2)=\sg{a,b}$ and $N_T(V_3)=\sg{c,b}$. 
  
  \begin{itemize}
      \item Assume first that $v$ is adjacent to $b$ or $c$ in $G$; say $b$ by symmetry (see \Cref{fig: H3}, left). Then, with $V_4:=\sg{a,c}$ and $V_5:=\sg{b,v}$, the minor obtained by contracting the $V_i$'s is a $K_5$, a contradiction.
      \item Assume now that $v$ is not adjacent to $b, c$, hence $|N_T(v)|\leq 1$.
      Since the edges $vw, cb$ belong to $R$, they are at distance $1$ from each other. In particular, since $w\in V_2$ is not adjacent to $c$, $w$ must be adjacent to $b$ in $G$.
      \begin{itemize}
          \item If $w$ is adjacent to $a$, we now set $V'_1:=\sg{u,v}, V'_2:=\sg{x,y}$ and $V'_3:=\sg{t,z}$ (see \Cref{fig: H3}, right). Then we are necessarily in one of the previous cases, with $w$ playing the role of $v$ and $V'_i$ playing the role of $V_i$ for each $i\in \sg{1,2,3}$.
          \item If $w$ is not adjacent to $a$, since both $vw$ and $wx$ are at distance $1$ from $ac$, both $v$ and $x$ are adjacent to $a$.
          We now set $V'_2:=\sg{v,w}$ and $V'_3:=\sg{x,y}$. Then $a$ is adjacent to $V_1$, $V_2'$ and $V'_3$, so we are in the same situation as in Case $1$, with $V'_i$ playing the role of $V_i$ for each $i\in \sg{2,3}$.
      \end{itemize}
  \end{itemize}
\end{proof} 
  
\begin{figure}[htb]
  \centering   
  \includegraphics[scale=1]{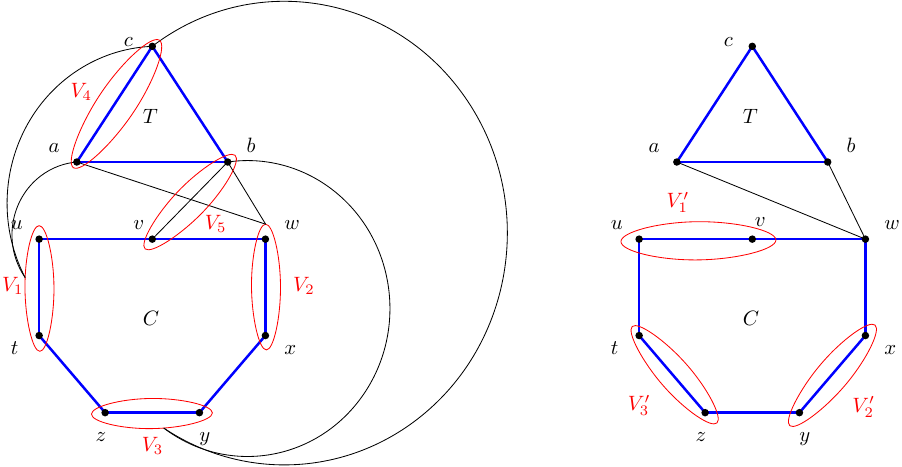}
  \caption{Proof of \Cref{lem: H3}. The edges of $R$ are represented in blue, while the black edges denote the edges from $G$.} 
  \label{fig: H3}
\end{figure}

\subsubsection{Excluding \texorpdfstring{$C_9$}.}
\begin{lemma}
 \label{lem: H4}
 If $G$ is a $K_5$-minor-free graph and $R$ is a neighbouring set of edges of $G$, then $G[R]$ does not contain $H_4:=C_9$ as a subgraph.
\end{lemma}

\begin{proof}
 We assume for a contradiction that $G[R]$ contains a cycle $C$ of length $9$, and denote its vertices with $a,b,c,d,e,f,g,h,i$, taken in a cyclic order.

\begin{claim}
\label{clm: C9_1}
    The vertex $a$ cannot be adjacent in $G$ to both $\sg{d,e}$ and $\sg{f,g}$.
\end{claim}
\begin{proof}
    The proof follows from the top left side of \Cref{fig: C9_12}: if we assume for a contradiction that $a$ is adjacent in $G$ to both an element of $\sg{d,e}$ and of $\sg{f,g}$, then after contracting the edges $bc, de, fg$ and $hi$ (that are pairwise adjacent in $G$ as they all belong to $R$), we obtain a $K_5$ minor in $G$. 
\end{proof}

\begin{figure}[htb]
    \centering   
    \includegraphics[scale=0.7]{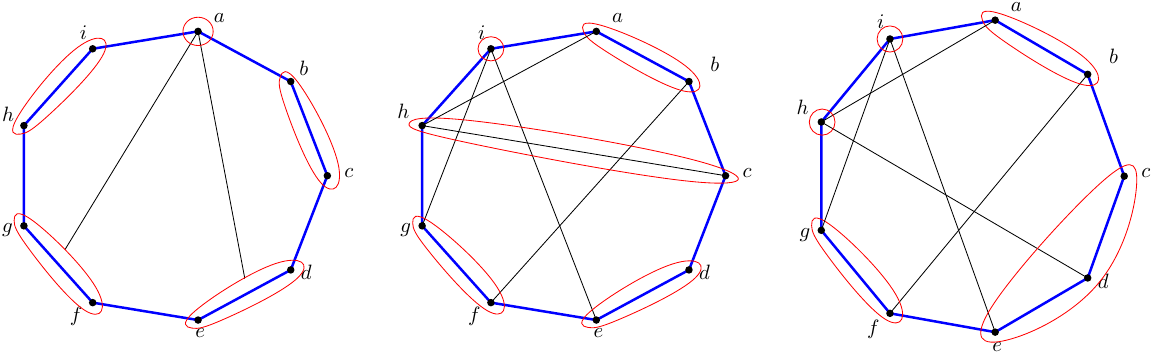}
    \caption{Proof of \Cref{clm: C9_1,clm: C9_2}. The edges of $R$ are represented in blue, while the black edges denote the edges from $G$.}   
    \label{fig: C9_12}
\end{figure}

\begin{claim}
   \label{clm: C9_2}
  We cannot have both $fb\in E(G)$ and $ei\in E(G)$.
\end{claim}

\begin{proof}
    We assume for a contradiction that both $fb$ and $ei$ are edges of $G$. Therefore, by \Cref{clm: C9_1} (with respectively $f$ and $e$ playing the role of $a$) , we get that $\{fi, fa, ea, eb\}\cap E(G)=\emptyset$. On the other hand, $R$ is a neighbouring set, and thus $\{a,b\}$ must be adjacent to $\{d,e\}$ and $\{a,i\}$ must be adjacent to $\{g,f\}$. Therefore, $da$ or $db$ is an edge in $G$ and similarly $ga$ or $gi$ is an edge of $G$. Moreover, $da$ and $ga$ cannot both exist by \cref{clm: C9_1}. The other case being symmetric, we may assume that $ga$ is not an edge, and thus that $gi$ is an edge in $G$.  Note that $\sg{a,b}$ must be adjacent to $\sg{h,g}$, and moreover, as we assumed that $ga$ is not an edge, and as by \Cref{clm: C9_1}, $b$ cannot be adjacent to $g$ nor $h$, $ah$ must be an edge in $G$. 
    Again, since $\{i,h\}$ must be adjacent to $\{c,d\}$, and $ie$ is an edge in $G$, by \Cref{clm: C9_1}, $i$ cannot be adjacent to $\sg{c,d}$, hence $hc$ or $hd$ must be an edge in $G$.
    If $hc$ is an edge, then after contracting edges $ab, de, fg$ and $hc$, we obtain a $K_5$ minor in $G$ (see \Cref{fig: C9_12}, top right).
    Thus $hd$ is an edge and then after contracting $\{a,b\}$, $\{c,d,e\}$, $\{g,f\}$ we obtain a $K_5$ minor in $G$ (see \Cref{fig: C9_12}, bottom).
\end{proof}

\begin{claim} \label{clm: C9_3}
    We cannot have both $be\in E(G)$ and $fi\in E(G)$.
\end{claim}
\begin{proof}
    Assume for a contradiction that both $be$ and $fi$ are edges of $G$.
    In particular, by two applications of \Cref{clm: C9_1}, with $i$ and $b$ playing the role of $a$ (up to performing a rotation of $C$), we obtain that $i$ cannot be adjacent to $\sg{c,d}$, and that $b$ cannot be adjacent to $\sg{g,h}$. On the other hand, since $R$ is neighbouring, there has to be an edge between $\sg{b,c}$ and $\sg{g,h}$, and thus it has to be either $cg$ or $ch$. With a symmetric argument, in order to connect the sets $\sg{h,i}$ and $\sg{c,d}$, either $hc$ or $hd$ is an edge in $G$.

    Suppose first that $ch\notin E(G)$.
    By above remarks, we then must have $hd, cg\in E(G)$. 
    Moreover, there has to be an edge connecting $ab$ and $fg$ and because of \cref{clm: C9_1}, this cannot be $ga$ nor $gb$ (with $g$ taking the role of $a$)  nor $fb$ (with $f$ taking the role of $a$), so it has to be $fa$. However, note that then, the existence of $fa$ and $cg$ implies that, up to performing a rotation of $C$, we are in the configuration forbidden by \Cref{clm: C9_2}, a contradiction.

    Suppose now that $ch\in E(G)$. Then the existence of respectively $ch$, $be$ and $fi$ implies by \cref{clm: C9_1} 
    that $c$ in not adjacent to $f$ nor to $g$, $b$ is not adjacent to $g$, and $f$ is not adjacent to $b$. Therefore, there is no edge between the sets $\sg{b,c}$ and $\sg{f,g}$, contradicting the fact that $bc,fg$ belong to a neighbouring set.
\end{proof}    
  
Given an edge $e$ of $G[R]$, we call its \emph{antipode} the unique vertex which is at distance $4$ on the $C$ from its two endvertices. As $R$ is neighbouring, every edge of $C$ must be adjacent to the two edges of $C$ incident to its antipode. In particular, up to rotating $C$, \Cref{clm: C9_1,clm: C9_2,clm: C9_3} imply that every edge must be adjacent in $G$ to its antipode.
  
\begin{claim} \label{clm: C9_4}
    The edges $ae$ and $cg$ cannot both exist in $G$.
\end{claim}
\begin{proof}
  The proof follows from of \Cref{fig: C9_4}: $d$ is the antipode of $ih$, and thus must be connected to $\sg{i,h}$. 
  If we assume for a contradiction that $ae,cg\in E(G)$, then the couple of edges $di,cg$, or $dh, ae$ (depending of whether $d$ is adjacent to $i$ or $h$), implies that, up to a rotation of $C$, we are in the configuration forbidden by \cref{clm: C9_2}, a contradiction.
\end{proof}

\begin{figure}[htb]
    \centering   
    \includegraphics[scale=0.7]{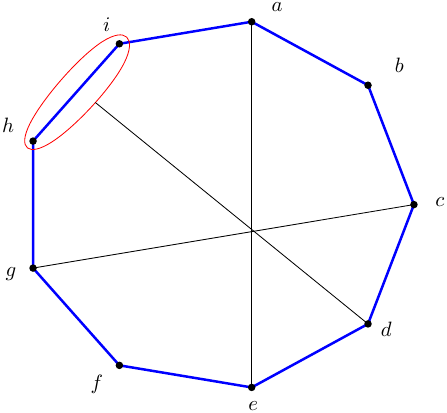}
    \caption{Proof of \Cref{clm: C9_4}. The edges of $R$ are represented in blue, while the black edges denote edges from $G$.} 
    \label{fig: C9_4}
\end{figure}

We now conclude the proof of \Cref{lem: H4}. Since $a$ and $c$ are respectively the antipodes of the edges $ef$ and $gh$, then $a$ must be adjacent to one of the vertices in $\sg{e,f}$ and $c$ must be adjacent to one of the vertices in $\sg{g,h}$.
Observe that we must then have $af, cg\in E(G)$, as all other possible pair of edges connecting $ef$ and $gh$ to their antipodes imply that, up to a rotation of $C$, we are in a configuration forbidden by \Cref{clm: C9_4}. However, note that the existence of these two edges implies that, up to a rotation of $C$, we are in a configuration forbidden by \cref{clm: C9_2}, a contradiction.
\end{proof}

We note that for each $i\in \sg{1,2,3,4}$, $H_i$ is edge-minimal, in the sense that for every proper subgraph $S$ of $H_i$, one can build a planar (and thus $K_5$-minor-free) graph $G$ that contains a neighbouring set $R$ such that $S$ is a subgraph of $G[R]$.


\subsection{Fractional matching}
\label{sec: proof-frac}

The aim of this section is to show that a neighbouring set of edges in a $K_5$-minor-free graph has fractional matching at most $\frac{9}{2}$ (see \Cref{thm: frac-intro}).
To do so, we first introduce some preliminary definitions and results about matchings.

Let $H$ be a graph with no isolated vertex, and let $M=\sg{e_1,\ldots, e_k}$ be a maximum matching of $H$, with $k:=\mu(H)$.
For each $i\in \sg{1,\ldots, k}$, we write $e_i=u_iv_i$, with $u_i, v_i\in V(G)$ and we denote by $S_M:=\sg{u_i, v_i : 1\leq i \leq k}$ the set of vertices saturated by $M$. 
 
A \emph{private triangle} of an edge $e_i\in M$ is a triangle of $H$ of the form $u_iv_iw$, for some vertex $w\in V(H)\setminus S_M$. 
Since $M$ is a maximum matching, every edge $e_i$ has at most one private triangle.

A \emph{private edge} of a vertex $u$ of $S_M$ (belonging to $e_i$ for some $1\leq i \leq k$) is an edge $uw\in E(H)$, with $w\in V(H)\setminus S_M$ which does not belong to a private triangle of $e_i$. Since $M$ is a maximum matching, for each $i\in \sg{1, \ldots, k}$, at most one endvertex of $e_i$ can have private edges. For the same reason, if $e_i$ has a private triangle, then none of its endvertices can have a private edge. In the remainder of the proof, we will follow the convention that whenever one of the two endvertices of the edge $e_i$ has a private edge, then it is always $u_i$.

Observe that some vertices of $V(H)\setminus S_M$ can belong to several private edges or private triangles. Moreover, every vertex of $V(H)\setminus S_M$ belongs to at least one private edge or private triangle (otherwise, since $H$ has no isolated vertex, we would obtain a contradiction with the maximality of $M$).

For a given maximum matching $M$, an \emph{$M$-partition} of $V(H)$ is a tuple $(U,V,W,X)$ where each of $U$, $V$, $W$ and $X$ are sets of vertices of $H$ with the following properties:

\begin{itemize}
    \item $(U,V)$ is a partition of $S_M$ with $U=\sg{u_i : 1\leq i \leq k}$ and $V=\sg{v_i : 1\leq i \leq k}$;
    \item $W$ is the set of vertices from $V(H)\setminus S_M$ that belong to some private triangle;
    \item $X$ is the set of vertices from $V(H)\setminus (S_M\cup W)$ that belong to a private edge  but not to a private triangle.
\end{itemize}   

In the following claim, we gather some properties of an $M$-partition, whose proofs immediately follow from the above definitions, and from the fact that the matching $M$ is given and maximum and $H$ has no isolated vertices.

\begin{claim}\label{proprietesM-partition}
  For every graph $H$ with no isolated vertex, for every maximum matching $M$ of $H$ and for every $M$-partition $(U,V,W,X)$ of $V(H)$:
    \begin{enumerate}
        \item $(W,X)$ is a partition of $V(H)\setminus S_M$, and thus $(U,V,W,X)$ is a partition of $V(H)$.
        \item $W\cup X=V(H)\setminus S_M$ is an independent set.
        \item  For every $v_i\in V$, $v_i$ is adjacent to some vertex $w$ in $V(H)\setminus S_M$ if only if $u_iv_iw$ is a private triangle. In particular, there is no edge between $X$ and $V$, and thus $N_H(X)\subseteq U$.
        \item If $(U',V',W',X')$ is another $M$-partition of $V(H)$, then $W'=W$ and $X'=X$, so the partitions can only differ on the partition $(U,V)$ and $(U',V')$ of $S_M$, \emph{i.e.}, on the choices of the parts of some $u_i,v_i$ that have symmetric roles (that is, if there is no private edge adjacent to $u_iv_i$). 
    \end{enumerate}
\end{claim}

For the results of this section, we are interested in the case when $H$ is edge-induced by a neighbouring set of edges in a $K_5$-minor-free graph, and in particular when $\mu(H)=4$. In this case, we get the following additional property of the $M$-partition. 

\begin{claim}\label{claimSizeWX}
  Let $G$ be a $K_5$-minor-free graph, $R$ a neighbouring set of edges of $G$, and let $H=G[R]$.
  If $\mu(H)=4$, then for every $M$-partition $(U,V,W,X)$ of $H$, $|W|\leq 1$. Furthermore, if $|V(H)|\geq 10$, then $|W \cup X|\geq 2$, and thus $|X|\geq 1$.
\end{claim}
\begin{proof}
  If $W$ contains two distinct vertices $w$ and $w'$, then there exist $i,j$ such that $u_iv_iw$ and $u_jv_jw'$ form two private triangles. As each edge has at most one private triangle, $i\neq j$. Since $|M|=4$, these two triangles together with the two remaining edges of the matching $M$ form a copy of $H_1$, contradicting \cref{lem: H1}. The ``furthermore'' part then immediately follows from item $1$ of \Cref{proprietesM-partition}.
\end{proof}

We say that a graph $H$ is \emph{tame} if $H$ has either at most $9$ vertices, or if it is the disjoint union of a cycle $C$ of length $7$ and a star of center $v$, with possibly additional edges having both endvertices in $V(C)\cup \sg{v}$.

Tame graphs can be easily seen to have a fractional vertex cover with value $\tfrac92$.
\begin{lemma}
    \label{clm:tameIsEasy}
    If $H$ is tame, then $\tau^*(H)\leq \tfrac92$.
\end{lemma}

\begin{proof}
    Observe that if $H$ has at most $9$ vertices, then it suffices to set $h(v):=\frac12$ for every vertex of $H$ to get a fractional vertex cover with value at most $\tfrac92$.
        
    Suppose now that $H$ is the disjoint union of a cycle $C$ of length $7$ and a star of center $v$, with possible additional edges between vertices of $V(C)\cup \sg{v}$. 
    Then all the edges of $H$ are either incident with $v$ or between two vertices of $C$.
    In this case, setting $h(v):=1$, $h(u):=\tfrac{1}{2}$ for every $u \in V(C)$, and $h(x):=0$ for the remaining vertices, produces a fractional vertex cover with value $\tfrac92$.
\end{proof}

Given a matching $M$ of a graph $H$, a path $P=x_1x_2 \ldots x_{k}$ is \emph{alternating} if $x_1 \in V(H) \setminus S_M$,
for every pair of subsequent edges $x_{i-1}x_i$, $x_ix_{i+1}$, one of them is in $M$ and the other one is not in $M$, and moreover, if $P$ contains a vertex of $S_M$, then $P$ also contains its neighbour in $M$. 

Furthermore, if $x_k \in V(H) \setminus S_M$, then we say that $P$ is \emph{augmenting}. Note that if $P$ is augmenting, then $k$ must be even, and both the first and last edges of $P$ are not in $M$.
The vertex $x_1$ is called the \emph{root} of $P$. Given an alternating path $P$, the operation of \emph{reversing $M$ through $P$} consists in constructing a new set of edges $M'$ of $V(H)$ by swapping the edges in $P$, that is: $M' := (M \setminus E(P)) \cup (E(P) \setminus M)$. Note that $P$ is still an alternating path for the constructed set $M'$, and note that moreover, if $P$ is augmenting, then $M'$ is also a matching in $H$. Moreover, note that $|M'|>|M|$, which immediately implies the following (folklore) claim.

\begin{claim}
    \label{clm:noAugmenting}
    If $M$ is a maximum matching of $H$, then there is no augmenting path.
\end{claim}

Given a matching $M$ of a graph $H$, we say that two alternating paths $P,P'$ are \emph{$S_M$-disjoint} if $V(P)\cap V(P')\cap S_M = \emptyset$. Note that if $M$ is a maximum matching, then, by \cref{clm:noAugmenting}, every alternating path has at most one vertex in $V(H)\setminus S_M$, hence in this situation, every two $S_M$-disjoint paths intersect in at most one vertex (which is an extremity of both paths).

\begin{lemma}\label{cl:alt path}
  Let $G$ be a $K_5$-minor-free graph and $R$ a neighbouring set of edges of $G$. 
  Let $H=G[R]$ and assume that $\mu(H)=4$, and let $M$ be a maximum matching in $H$.

  Let $P,P'$ be two $S_M$-disjoint alternating paths, and let $x,v$ (resp. $x',v'$) be the extremities of $P$ (resp. $P'$) such that $x$ (resp. $x'$) is the root of $P$ (resp. $P'$).
  Then either $H$ is tame, or $v$ and $v'$ are not adjacent in $H$.
\end{lemma}

\begin{proof}
  By \Cref{clm:noAugmenting}, both $P$ and $P'$ are not augmenting, and thus as they are $S_M$-disjoint, $P\setminus \sg{x}$ and $P'\setminus \sg{x'}$ are vertex-disjoint.
  Assume by contradiction that $v$ and $v'$ are adjacent in $H$ and that $H$ is not tame; in particular $|V(H)|\geq 10$.
  If $x\neq x'$, then we get an augmenting path between $x$ and $x'$ (by following $P$ from $x$ to $v$, taking $vv'$ and then following $P'$ from $v'$ to $x'$, contradicting \Cref{clm:noAugmenting}.
  Thus $x=x'$ and the two paths plus the edge $vv'$ form a cycle $C$ of odd length, and since $\mu(H)=4$, its length is either $5$, $7$ or $9$ (the assumption that $P$ and $P'$ are $S_M$-disjoint implies that $C$ cannot have length $3$). 
  By \cref{lem: H4}, $H$ cannot contain a $C_9$ as subgraph, so $C$ can only have length $5$ or $7$.
  
  If $C$ has length $5$, then $C$ and the remaining two edges of $M$ form $H_2$ as a subgraph of $H$, contradicting \cref{lem: H2}.
   
  Therefore $C$ has length $7$, and we may assume that $u_4v_4$ is the edge of $M$ that does not belong to $C$.
  Since $H$ is not tame, we get by \cref{claimSizeWX} that $|V(H)\setminus S_M|\geq 2$ thus $V(H)\setminus (S_M \cup \sg{x})$ is non-empty. Let $x''$ be a vertex in $V(H)\setminus (S_M \cup \sg{x})$. If $x''$ is adjacent to both $u_4$ and $v_4$, then we get $H_3$ as subgraph, contradicting \cref{lem: H3}. If  $x''$ is adjacent to $V(C)\setminus\sg{x}$, then we get an augmenting path between $x$ and $x''$, contradicting \Cref{clm:noAugmenting}. If $x''$ is adjacent to some other vertex $y$ of $V(H)\setminus (S_M)$, then $M\cup\sg{xy}$ forms a matching of $H$, contradicting the maximality of $M$.
  Thus, $V(H)\setminus (S_M \cup \sg{x})$ is an independent set of vertices which can only be adjacent to $u_4$ (recall our convention that $v_4$ cannot have a private edge). 
  Moreover, if $v_4$ is adjacent to $C$, then we get an augmenting path between some vertex $x'' \in V(H)\setminus (S_M \cup \sg{x})$ and $x$, contradicting \Cref{clm:noAugmenting}. 
  Therefore $H$ is tame, a contradiction.
\end{proof}

\begin{lemma}\label{cl: path collection}
  Let $G$ be a $K_5$-minor-free graph and $R$ a neighbouring set of edges of $G$. 
  Let $H=G[R]$ and assume that $\mu(H)=4$, that $H$ is not tame, and let $M$ be a maximum matching in $H$.
 
  Let $\mathcal{P}$ be a family of pairwise $S_M$-disjoint alternating paths which do not intersect any private triangle, and which maximizes $|S_M\cap V(\mathcal P)|$.
  There exists an $M$-partition $(U,V,W,X)$ such that for each path $P\in \mathcal P$ with root $x\in X$, all the vertices at even (resp. odd) distance from $x$ along $P$ are in $V$ (resp. $U$). 
 
 Moreover, there is no edge from $X$ to $S_M\setminus V(\mathcal{P})$.
 \end{lemma}

\begin{proof}
   Since $H$ is not tame, it has order at least~10, and thus by \Cref{claimSizeWX}, $X\neq \emptyset$.

   Assume that $\mathcal P$ is a family of pairwise $S_M$-disjoint alternating paths which do not intersect any private triangle, and which maximizes $|S_M\cap V(\mathcal P)|$.
   Note that every edge $uv\in M$ intersects at most one path of $\mathcal P$, and that if some path $P\in \mathcal P$ intersects $uv$, then $uv$ is an edge of $P$. 
   Let $P$ be such a path with root $x\in X$, and assume without loss of generality that there is an edge $uv\in M\cap E(P)$, where 
   $u$ is at even distance from $x$ on $P$, while
   $v$ is at odd distance from $x$ on $P$. Note that as $P$ is alternating, $v$ must be the predecessor of $u$ on $P$.  
   
   We claim that $u$ cannot have a private edge $ux'$. Assume for a contradiction that $u$ is incident with such an edge. Note that we must then have $x=x'$, as otherwise, there would be an augmenting path from $x$ to $x'$.
   Moreover, $v$ is not adjacent to $x$, as otherwise $uvx$ would form a private triangle intersecting $P$. 
   Let $P'$ be the restriction of $P$ from $x$ to the predecessor of $v$ along $P$: $P'$ and $xuv$ form two $S_M$-disjoint paths such that their extremities different from $x$ are adjacent, contradicting \Cref{cl:alt path}.
   
   We thus proved that every vertex of some path $P\in \mathcal P$ having a private edge must be at odd 
   distance on $P$ from its root. It then implies the existence of an $M$-partition with the desired properties. 
   
  To see that there is no edge from $X$ to some vertex from $S_M\setminus V(\mathcal P)$, note that every such edge would be of the form $yu_i$ for some $y\in X$ and $i\in \sg{1, \ldots, 4}$ such that $u_iv_i$ is not covered by some path of $\mathcal P$. In particular, note that then $yu_i$ should be a private edge of $u_i$, implying that $u_iv_i$ has no private triangle. Therefore, the family $\mathcal P\cup \sg{yu_iv_i}$ is a family of pairwise $S_M$-disjoint alternating paths that avoids private triangles, contradicting the maximality of~$|S_M\cap V(\mathcal P)$.
\end{proof}

We are now ready to prove the main result of this section.

\fractionalMatching*

\begin{proof}
 Let $H:=G[R]$. In particular, $H$ has no isolated vertex.
 We will use the fact that $\mu^{\ast}(H)=\tau^{\ast}(H)$, and construct a fractional vertex-cover $h$ of $H$ with value at most $\tfrac92$. 

 First, observe that we must have $\mu(H)\leq 4$, as otherwise, since $R$ is a neighbouring set, the contraction of any matching of size $5$ in $H$ would induce a minor isomorphic to $K_5$ in $G$. 
 On the other hand, recall that the set of endvertices of any maximum matching in $H$ forms a vertex-cover of $H$, thus $\tau(H)\leq 2\mu(H)$. Therefore, as $\mu^{\ast}(H) \leq \tau^{\ast}(H)\leq \tau(H)$, if $\mu(H) \leq 2$, then $\mu(H)^{\ast} \leq 4$ and we are done.
 Hence the remaining cases are $\mu(H)=3$ and $\mu(H)=4$.

\medskip

\paragraph{Case 1: $\mathbf{\mu(H)=3}$.\\}

 To prove that $\mu^{\ast}(H)\leq \frac{9}{2}$, we construct a fractional vertex cover $h$ of $H$ as follows. Let $M$ be a maximum matching of $H$ and $(U,V,W,X)$ an $M$-partition of $V(H)$.
 For every $i\in \sg{1,2, 3}$, we set the following:

 \begin{itemize}
  \item if $e_i$ has a private triangle $u_iv_iw_i$, then we set $h(u_i):= h(v_i) :=h(w_i) := \tfrac12$; 
  \item otherwise, we set $h(u_i):=1$ and $h(v_i):=\tfrac12$.
 \end{itemize} 
 For every remaining vertex $x$, we set $h(x)=0$.
 Note that $h$ is well-defined since this covers all possible cases of the $M$-partition. 
 
 We first prove that $h$ is a fractional vertex-covering of $H$. Let $e=xy\in R$, we will show that $h(x)+h(y)\geq 1$.
 First, note that for every vertex $u\in U\cup V\cup W$, we have $h(u)\geq \tfrac12$, hence if both extremities of $e$ belong to $U\cup V\cup W$, we get $h(x)+h(y)\geq 1$. Consider without loss of generality that $x$ is not in this set, i.e.~$x\in X$. Then, since $N(X)\subseteq U$ (by \cref{proprietesM-partition}), we must have that $y=u_i$ for some $i$. Since $u_i$ has a private edge, $u_i$ does not belong to a private triangle, so $h(u_i)=h(y)=1$. Therefore $h(x)+h(y)\geq h(y)\geq 1$ and we are done.

 It remains to show that $h$ has value at most $\frac{9}{2}$. Let $a$ denote the number of edges in $M$ having a private triangle. Note that the only vertices of $H$ with non zero value are the vertices from $U\cup V \cup W$. Moreover, as every edge from $M$ has at most one private triangle $T$ with total value $\tfrac32$, and as every edge from $M$ which does not have a private triangle has total value at most $\tfrac32$, we then obtain $h(V(H))\leq a\cdot\tfrac32 +(|M|-a)\tfrac32=\tfrac92$.

\paragraph{Case 2: $\mathbf{\mu(H)=4}$.\\}

If $H$ is tame, then we conclude immediately thanks to \Cref{clm:tameIsEasy}. Hence, in the following, we can assume that $H$ is not tame. 
Let $M$ be a maximum matching of $H$ that maximizes the size of $W$ in any $M$-partition $(U,V,W,X)$ (this is well defined since the size of $W$ does not depend on the choice of the $M$-partition, as observed in \cref{proprietesM-partition}).
Let $\mathcal{P}$ be a collection of pairwise $S_M$-disjoint alternating paths that do not intersect any private triangle and that maximizes $|S_M\cap V(\mathcal P)|$. If there are several such matchings for which $|W|$ has the same size, then we choose such a matching $M$ that moreover maximizes $|S_M\cap V(\mathcal{P})|$.
Let $(U,V,W,X)$ be a $M$-partition given by \cref{cl: path collection}.

\begin{itemize}
  \item For each $u_i$ in $U\cap V(\mathcal{P})$, we set $h(u_i):= 1$;
  \item For each $v_i$ in $V\cap V(\mathcal{P})$, we set $h(v_i):= 0$;
  \item For each $x$ in $X$, we set $h(x):= 0$;
  \item For each remaining vertex $v$ ($v\in W \cup S_M\setminus V(\mathcal{P})$), we set $h(v):= \tfrac12$.
 \end{itemize} 

 We first check that $h(V(H))\leq \tfrac{9}{2}$.
 We have by definition $h(X)=0$. Since $|W| \leq 1$, $h(W)\leq \tfrac12$. Finally, for each edge $u_iv_i \in M$, $h(u_i)+h(v_i)=1$, hence $h(U)+h(V)=4$. Thus, in total, $h(V(H)) = h(U)+h(V)+h(W)+h(X)\leq \tfrac{9}{2}$.

 \medskip

 We now show that $h$ is a fractional vertex cover, that is, we prove that there is no edge $yz$ such that $h(y)=0$, and $h(z)\leq \tfrac{1}{2}$. Assume by contradiction that such an edge exists. Since $h(y)=0$, we have $y\in X\cup (V\cap V(\mathcal{P}))$.
 
 Suppose first that $y\in X$. 
 Since, by \cref{proprietesM-partition}, $N_H(X)\subseteq U$, we have $z\in U \subset S_M$. Moreover, since $h(z)\leq \tfrac{1}{2}$ and by the rules of construction of $h$, we have $z \not\in V(\mathcal{P})$. Hence, there is an edge between $S_M \setminus V(\mathcal{P})$ and $X$, contradicting \cref{cl: path collection}. 
 Therefore, $y$ cannot be in $X$.
 
 We now assume that $y\in V\cap V(\mathcal{P})$, and write without loss of generality $y=v_1$. We consider $P\in \mathcal{P}$ such that $v_1\in V(P)$, and let $x\in X$ be its root.
 Note that if $z\in X\cup W$, the adjacency between $v_1\in V$ and $z$ implies, by \cref{proprietesM-partition}, that $u_1v_1z$ is a private triangle, contradicting the fact that $v_1\in V(\mathcal{P})$. Therefore, $z$ belongs to $S_M$. 
 We now distinguish several cases.

 \begin{enumerate}
  \item Assume first that $z \in V(P)$, and note that then, as we assumed that $h(z)\leq \tfrac12$, we have $z\in V$. We thus write without loss of generality $z=v_2$. Note that since $P$ is alternating and both $y$ and $z$ are in $V\cap V(P)$, the path $P$ starting from $x$ does not contain the edge $v_1v_2$. Recall that by \Cref{cl: path collection}, $v_1$ and $v_2$ are at even distance from $x$ on $P$. Hence, the subpath of $P$ connecting $v_1$ and $v_2$ together with the edge $v_1v_2$ creates a cycle $C$ of $H$ with length either $3,5$ or $7$ (note that by \Cref{lem: H4}, $H$ does not contain $C_9$ as a subgraph). In what follows, we may assume without loss of generality that $v_2$ is closer from $x$ than $v_1$ on $P$.
  \begin{enumerate}
    \item If $C$ has length $3$, then its three vertices are $v_2$, $u_1$, $v_1$. If $W \neq \emptyset$, then some edge of $M$ has a private triangle $T$. By definition of $\mathcal P$, $T$ must be disjoint from $P$, hence  observe that in this case, since $M$ has four edges, we can find a subgraph of $H$ isomorphic to $H_1$, obtained by taking the triangles $T$ and $C$, together with some edges in the symmetric difference $M\Delta E(P')$, where $P'$ is the subpath of $P$ from $x$ to $v_2$, contradicting \cref{lem: H1}. 

    If $W = \emptyset$, then, we can reverse $P$ from $x$ to $v_1$.
    We thus get a new matching containing $v_2u_1$ with private triangle $v_2 u_1 v_1$, contradicting our previous assumption that $M$ maximizes $|W|$.
    
    \item If $C$ has length $5$, and since $H$ is not tame, then we can find a subgraph of $H$ isomorphic to $H_2$, by taking $C$ together with the edges in the symmetric difference $M\Delta E(P')$, where $P'$ is the subpath of $P$ from $x$ to $v_2$, contradicting \cref{lem: H2}.
    
    \item If $C$ has length $7$, then $P$ must contain all the edges from $M$. In particular, $W$ is empty. 
    As $H$ is not tame, we cannot have all the vertices of $V(H)\setminus S_M=X$ (which has size at least two) adjacent to only $u_2$. Thus, and since $N_H(X)\subseteq U$, there must exist $\ell \neq 2$ and two distinct vertices $x',x''$ in $X$ ($x$ can be either of them) such that $x'$ is adjacent to $u_2$ and $x''$ is adjacent to $u_\ell$.
    a vertex $x'\in V(H)\setminus S_M=X$, which is not adjacent to $u_2$
    We then get an augmenting path between $x'$ and $x''$, contradicting \Cref{clm:noAugmenting}.
   \end{enumerate}

   \item Assume now that $z \in V(P')$, for some $P' \in \mathcal{P}$ and $P'\neq P$. Then again, we may assume without loss of generality that $z=v_2$, and we denote by $x'$ the root of $P'$ in $X$. We obtain a contradiction with \cref{cl:alt path} when considering the two subpaths of $P$ and $P'$ going from $x$ (resp. $x'$) to $v_1$ (resp. $v_2$).

   \item Assume that $z \not\in V(\mathcal{P})$. Without loss of generality, we have $z\in \sg{u_4,v_4}$. Note that an alternating path contains $u_4$ if and only if it contains $v_4$, hence we must have $\sg{u_4, v_4}\subseteq S_M\setminus V(\mathcal P)$. In particular, by \Cref{cl: path collection} there is no edge between $X$ and $\sg{u_4, v_4}$. If the vertex $z'$ from $\sg{u_4,v_4}$ which is distinct from $z$ is adjacent to some $x'\in W$, then in particular, $x'\neq x$ and we obtain an augmenting path from $x$ to $x'$ by following the subpath of $P$ going from $x$ to $y$, and then the edges $yz$, $zz'$ and $z'x'$, contradicting \Cref{clm:noAugmenting}. 
   It follows that among $\sg{u_4,v_4}$, only $z$ can have a private edge (which, if it exists, must connect $z$ with a vertex in $W$). Hence, $z=u_4$ and $u_4v_4$ does not have a private triangle.

   If $v_1$ is an extremity of $P$, then note that after replacing $P$ in $\mathcal P$ by the path that extends it with $v_1u_4v_4$, we obtain a family $\mathcal P'$ of pairwise $S_{M}$-disjoint alternating paths that avoid all private triangles of edges of $M$ contradicting the fact that the family $\mathcal P$ maximizes $|S_M\cap V(\mathcal P)|$. 
   Hence we may assume in what follows that $u_1v_1u_2v_2$ is a subpath of $P$ (recall that by \Cref{cl: path collection}, the vertex $u_1$ is closer to $x$ than $v_1$ on $P$). We let $P_1$ denote the subpath of $P$ going from $v_1$ to $x$, and consider the matching $M'$ obtained by reversing $M$ through $P_1$. Let $(U',V',W',X')$ be an $M'$-partition.
   
   In particular, we now have $x\in S_{M'}$ and $v_1\in X'\cup W'$. As $P_1$ is disjoint from the private triangles of the edges from $M$, every private triangle of an edge of $M$ must be a private triangle of an edge of $M'$. In particular, it implies that $W\subseteq W'$, and thus, as we chose $M$ that maximizes $|W|$, we then have $W'=W$.
   Thus, $X'=(X\setminus\sg{x})\cup \sg{v_1}$, $ S_{M'}=(S_{M'}\setminus\sg{v_1})\cup \sg{x}$, $W'=W$ and if $W$ is not empty, the only possible private triangle with respect to both edges from $M$ and $M'$ is of the form $wu_3v_3$ for some $w\in W$. 

    We now want to define a family $\mathcal{P}'$ of pairwise $S_{M'}$-disjoint alternating paths that avoid all private triangles of edges of $M'$ such that $V(\mathcal{P}')\cap S_M=(V(\mathcal{P})\cap S_M)\cup \{u_4,v_4\}$.  Denote by $P_2$ the subpath of $P$ going from $v_1$ to its extremity which belongs to $S_M$, and $P_4:=v_1u_4v_4$. Observe that $\{P_1,P_2,P_4\}$ are pairwise $S_{M'}$-disjoint alternating paths rooted in $v_1$, and from the previous discussion, they avoid all private triangles of edges of $M'$.
    We define $\mathcal{P}'$ depending on the place of $u_3v_3$ in $\mathcal P$ as follow:

    \begin{enumerate}
        \item Assume that $u_3v_3$ is not covered by the paths of $\mathcal P$. Then it suffices to take $\mathcal P':=\sg{P_1, P_2, P_4}$. 
        
        \item Assume that $u_3v_3$ belongs to either $P_1$ or $P_2$. Again, it suffices to take $\mathcal P':=\sg{P_1, P_2, P_4}$.
        
        \item Assume now that $u_3v_3$ is covered by the paths of $\mathcal P$, but does not belong to either $P_1$ nor $P_2$. In particular, $W=W'=\emptyset$ so every alternating pas avoids trivially any private triangle, $P_1=v_1u_1x$, $P_2=v_1u_2v_2$, and there exists a path $P_3=x'u_3v_3\in \mathcal{P}$.
        
        If $x' \neq x$, then $x'\in X'$, and $P_3$ is also $S_{M'}$-disjoint from $P_1,P_2,P_4$, so we can take $\mathcal P':=\sg{P_1, P_2, P_3, P_4}$. 
        
        Finally, if $x=x'$, then we denote by $P'_1$ the path obtained from concatenating $P_1$ and $P_3$. Note that $P'_1$ is an alternating path for $M'$, rooted in $v_1$, which is $S_{M'}$-disjoint from $P_2,P_4$. We can thus take $\mathcal P':=\sg{P'_1, P_2, P_4}$.
    \end{enumerate}

    In each case, $\mathcal{P}'$ is a family of pairwise $S_{M'}$-disjoint alternating paths that avoid all private triangles of edges of $M'$, and such that $V(\mathcal{P}')\cap S_M=(V(\mathcal{P})\cap S_M)\cup \{u_4,v_4\}$, so $|V(\mathcal{P})\cap S_M|<|V(\mathcal{P}')\cap S_{M'}|$, which contradicts our choice of $(M, \mathcal P)$.
   
 \end{enumerate}

Putting all cases together, we deduce that no edge $yz$ with both $h(y)=0$ and $h(z)\leq \tfrac{1}{2}$ exists, hence $h$ is indeed a fractional vertex cover of $H$.
This concludes our proof of Case 2. 
\end{proof}

If we go in more details through all different cases of the proof of \Cref{thm: frac-intro}, we moreover obtain the following structural characterisation of the three extremal cases for which the bound $\tfrac92$ is attained, see \Cref{fig: minimal_92} for an illustration of those.

\corollaryOfMatching*

\begin{figure}[htb]
  \centering   
  \includegraphics[scale=0.7, page=2]{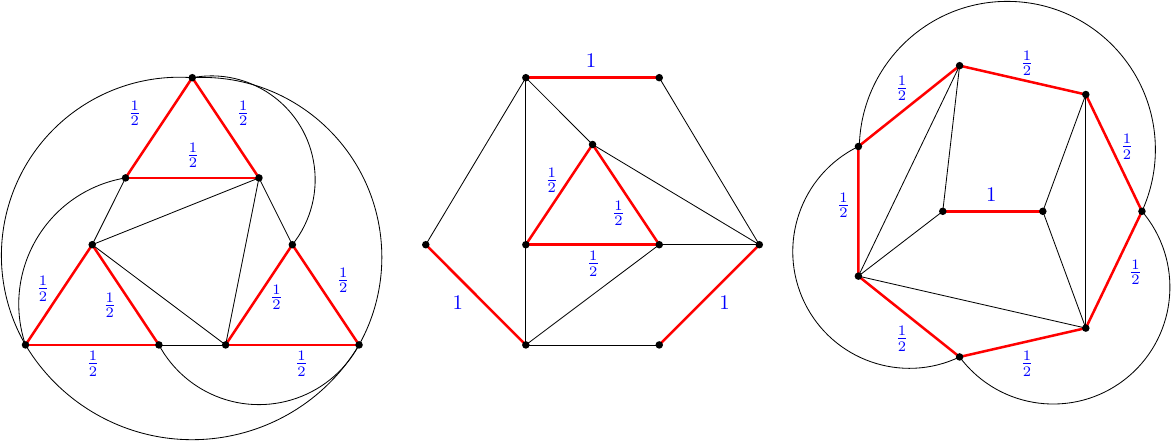}
  \caption{In red, the minimal three graphs contained in every neighbouring set of a planar graph attaining the bound $\tfrac92$ from \Cref{thm: frac-intro}. The black edges witness that the three graphs can be found as neighbouring sets of planar graphs. In blue some fractional matchings with value $\tfrac92$.} 
  \label{fig: minimal_92}
\end{figure}

\begin{proof}[Proof (sketch)]
We briefly sketch how to derive \Cref{cor: minimal_92} from our proof of \Cref{thm: frac-intro}. We assume that $H:=G[R]$ does not contain any of the graphs $3K_3, K_3+3K_2$ and $C_7+K_2$ as a subgraph, and show that then, $\mu^{\ast}(H)\leq 4$.
 We reuse the notations from the proof of \Cref{thm: frac-intro}. In the case $\mu(H)\leq2$, we have $\mu^{\ast}(H)\leq \tau(H)\leq 4$, hence we may assume that $\mu(H)\in \sg{3,4}$. Recall that in particular, the only possible vertices from $S_M$ that can have pendant edges are the $u_i$'s.
 
 Note also that in the proof of \Cref{thm: frac-intro}, in the case $\mu(H)=4$, the only cases in which we constructed a fractional vertex cover with value exactly $\tfrac92$ were either when $H$ is tame, or when one of the edges of a maximum matching belongs to a private triangle, the latter case implying immediately that $K_3+3K_2$ is a subgraph of $H$.
 
 \begin{claim}
  \label{clm: coro-9sommets}
  If $\mu(H)=4$ and $|V(H)|\leq 9$, then $\mu^{\ast}(H)\leq 4$.
 \end{claim}
 
 \begin{proof}
  Let $M$ be a maximum matching of $H$ and $(U,V,W,X)$ an $M$-partition of $V(H)$.
  If $|V(H)|\leq 8$, the result is immediate as one can assign weight $\tfrac{1}{2}$ to every vertex. So we assume that $|V(H)|=9$ and that the only vertex of $H$ not in $S_M$, say $v$, satisfies $d_H(v)\geq 1$.
  By previous remark, no edge of $M$ has a private triangle. In particular, $v$ is only adjacent to vertices of $U$. We let $U_1:=N_H(v)$,
  and $V_1:=\sg{v_j: u_j\in U_1}$.
  Note that $V_1$ is an independent set in $H$, as otherwise, $H$ would contain a copy of $H_2$, contradicting \Cref{lem: H2}.
  
  We now set $U_2:=N_H(V_1)\setminus (U_1\cup V_1)$. If $U_2=\emptyset$, then setting $h(x):=1$ for each $x\in U_1$, $h(v):=h(w):=0$ for each  $w\in V_1$, and  $h(y)=\tfrac{1}{2}$ for each $y\in V(H)\setminus  (U_1\cup V_1\cup\sg{v})$ gives a fractional vertex cover with value 4. We now assume that $U_2\neq\emptyset$. Note that $v\notin U_2$, as otherwise $H$ would contain $K_3+3K_2$. Observe that every edge of $M$ has at most one endvertex in $U_2$, as otherwise $H$ would contain either a copy of $K_3+3K_2$ or of $C_7+K_2$. Therefore, since none of the vertices of $U_2$ can have private edges,
  without loss of generality, we may assume that $U_2\subset U$, and set $V_2:=\sg{v_j: u_j\in U_2}$.
  Note that by definition $V_2$ is disjoint from $V_1$. Moreover, $V_2$ is an independent set of $H$, as otherwise $H$ would contain a copy of $K_3+3K_2$ or of $H_2=C_5+2K_2$ or of $H_4=C_9$,
  the latter two cases contradicting \Cref{lem: H2} and \Cref{lem: H4}.
  Moreover, there cannot be an edge between $V_1$ and $V_2$, as otherwise $H$ would contain either a copy of $K_3+3K_2$, or of $C_7+K_2$. Thus $V_1\cup V_2$ is an independent set in $H$.

  We similarly define for $i\in \sg{3, 4}$, $U_i:=N_H(V_{i-1})\setminus (U_1\cup V_1\cup \cdots \cup U_{i-1}\cup V_{i-1})$ (we stop if $U_i=\emptyset$) and, for the same reasons observe that we may assume that $U_i$ only contains vertices of $U$, and $V_i:=\sg{v_j: u_j\in U_i}$. Eventually, we claim that $\sg{v}\cup V_1\cup V_2\cup V_3\cup V_4$ is an independent set in $H$, as otherwise $H$ would contain one of the following graphs: $K_3+3K_2, H_2=C_5+2K_2, C_7+K_2, C_9$. Note also that our inductive definition implies that every vertex which is not in one of the sets $U_i\cup V_i$ cannot be adjacent to a vertex from $V_1\cup V_2\cup V_3\cup V_4$.
  
  We now define a fractional vertex cover of $H$ with value $4$ by assigning value $1$ to every vertex from $U_1\cup U_2\cup U_3\cup U_4$, value $0$ to every vertex from $\sg{v}\cup V_1\cup V_2\cup V_3\cup V_4$, and value $\tfrac12$ to all other vertices. We claim that the obtained map indeed defines a fractional vertex cover. 
 \end{proof}

 \begin{claim}
  \label{clm: coro-tame}
  If $H$ is tame and $\mu(H)=4$, then $\mu^{\ast}(H)\leq 4$.
 \end{claim}
 
 \begin{proof}
  By \Cref{clm: coro-9sommets}, we may assume that $|V(H)|\geq 10$, and that $H$ consists of the disjoint union of a cycle $C$ of length $7$ and a star with center $v$, with additional potential edges only between vertices of $V(C)\cup \sg{v}$. In particular, as $|V(H)|\geq 10$, note that the star with center $v$ has at least one edge, implying that $C_7+K_2$ is a subgraph of $H$.
 \end{proof}
 
 By \Cref{clm: coro-tame} and our previous  observation, we may assume from now on that $\mu(H)=3$. In particular, as $H$ has no isolated vertex, the following claim then immediately follows.
 
  \begin{claim}
  \label{clm: C7-minimaux}
 If $H$ has a copy of $C_7$ as a subgraph, then $|V(H)|=7$.
 \end{claim}

 Using \Cref{clm: C7-minimaux}, we may assume from now on that $H$ has no cycle of length seven.
  
 \begin{claim}
  \label{clm: C5-minimaux}
  If $H$ contains a copy of $C_5+K_2$, then it has a fractional vertex cover with value $4$.
 \end{claim}
 
 \begin{proof}
 Assume that $H$ has a copy of $C_5+K_2$, and let $a,b,c,d,e,f,g$ denote the pairwise distinct vertices of $H$, such that $a,b,c,d,e$ form a copy $C$ of $C_5$, and $fg\in E(H)$.
 Assume first that $fg$ belongs to a triangle $T$ of $H$ disjoint from $C$. Then, as $\mu(H)=3$, no other vertex from $V(C)\cup V(T)$ can have a neighbour outside $V(C)\cup V(T)$, and we obtain that $|V(H)|=8$, so $\mu^{\ast}(H)\leq 4$. 
 
 Assume now that $fg$ does not belong to a triangle disjoint from $C$. Then, as $\mu(H)=3$, only one of the vertices from $Y:=\sg{a,b,c,d,e,f,g}$ can have a neighbour outside $Y$ (call it $w$), and moreover, it should be either $f$ or $g$, say $f$. We then find a vertex cover $h$ with value $4$ by setting $h(f):=1$, $h(w):=0$ and  $h(x):=\tfrac12$ for each $x\in X\setminus\sg{f}$.
 \end{proof}

 Using \Cref{clm: C5-minimaux}, we may assume from now on that $H$ contains no copy of $C_5+K_2$. 

 We now claim that, using the fact that $\mu(H)=3$, the proof of \Cref{cl:alt path} generalizes and gives the following. 
 
 \begin{claim}[Analogue of \Cref{cl:alt path}]
  \label{clm: coro-L13}
  Let $P,P'$ be two $S_M$-disjoint alternating paths, and let $x,v$ (resp. $x',v'$) be the extremities of $P$ (resp. $P'$) such that $x$ (resp. $x'$) is the root of $P$ (resp. $P'$).
  Then either $H$ is tame, or it contains a copy of $C_7$ or $C_5+K_2$, or $v$ and $v'$ are not adjacent in $H$.
 \end{claim}

 \begin{claim}
  \label{clm: coro-2triangles}
  If at least two edges from $M$ have private triangles, then $\mu^{\ast}(H)\leq 4$. In particular, if $H$ contains $2K_3+K_2$ as a subgraph, then $\mu^{\ast}(H)\leq 3$. 
 \end{claim}
 
 \begin{proof}
  Assume without loss of generality that $e_1$ and $e_2$ belong to private triangles, and denote them respectively by $T_1, T_2$. First, note that if $T_1$ and $T_2$ intersect, then the fractional matching defined in the proof of \Cref{thm: frac-intro} has value at most $4$, hence we may assume that $T_1$ and $T_2$ are disjoint. Moreover, if $e_3$ has a private triangle $T_3$, either $T_3$ is disjoint from $T_1\cup T_2$, in which case $H$ contains $3K_3$, or $T_3$ intersects one of the triangles $T_1, T_2$, in which case we might conclude symmetrically that $\mu^{\ast}(H)\leq 4$. 
 
  We thus may assume that $T_1$ and $T_2$ are disjoint, and that $e_3$ has no private triangle. Note that if all neighbours of $u_3,v_3$ are in $T_1\cup T_2$, we can find a fractional vertex cover with value $4$ by giving values $\tfrac12$ to $u_3, v_3$ and to the vertices of $T_1\cup T_2$.
  
  Assume now that $u_3$ has a pendant edge $e=u_3w_3$, such that $w_3$ is not in $T_1\cup T_2$. Then we claim that assigning weights $\tfrac12$ to vertices from $T_1\cup T_2$, weight $1$ to $u_3$ and weights $0$ to all other vertices defines a fractional vertex cover of $H$ with value $4$. This comes from the observation that any edge between $v_3$ and some vertex different from $u_3$ would contradict that $\mu(H)=3$. 
\end{proof}

Using \Cref{clm: coro-2triangles}, we may assume from now on that $H$ contains no copy of $2K_3+K_2$. In particular, note that if $|V(H)|\geq 7$, then it implies that $|X|\geq 1$. We then claim that using \Cref{clm: coro-L13}, the proof of \Cref{cl: path collection} immediately generalizes and imply the following.

\begin{claim}[Analogue of \Cref{cl: path collection}]
\label{clm: coro-L14}
  Assume that $H$ is not tame, and let $\mathcal{P}$ be a family of pairwise $S_M$-disjoint alternating paths which do not intersect any private triangle, and which maximizes $|S_M\cap V(\mathcal P)|$.
  There exists an $M$-partition $(U,V,W,X)$ such that for each path $P\in \mathcal P$ with root $x\in X$, all the vertices at even (resp. odd) distance from $x$ along $P$ are in $V$ (resp. $U$). 
 Moreover, there is no edge from $X$ to $S_M\setminus V(\mathcal{P})$.
\end{claim}

We now claim that we can reproduce the proof from the case $\mu(H)=4$ of \Cref{thm: frac-intro}: we choose $M$ and the family $\mathcal P$ exactly the same way, using \Cref{clm: coro-L14}, and define the mapping $h: V(G)\to \mathbb R^+$ as before. By \Cref{clm: coro-2triangles}, $|W|\leq 1$, hence it follows that $h$ has value at most $\tfrac72$.

To show that $h$ defines a fractional vertex-cover, we claim that again, the case disjunction from the proof of \Cref{thm: frac-intro} works exactly the same, with the slight difference that we use in Case 1.a) that $H$ does not contains $2K_3+K_2$ (instead of $H_1$), in Case 1.b) that $H$ does not contains $C_5+K_2$, and in Case 1.c) that $H$ does not contains $C_7$. The proofs of Cases 2 and 3 then generalizes the exact same way, using \Cref{clm: coro-L13,clm: coro-L14}.
\end{proof}

\section{Further discussion}
\label{sec: ccl}
We strongly believe that the ideas behind the reduction from the diameter problem to a fractional matching problem could be applied to a more general setting than the computation of $\np_{\Delta, 3}$. In particular, it could be interesting to try reusing them to compute more general bounds on $\np_{\Delta, D}$ when $D$ is odd (recall that the cases where $D$ is even have been solved in \cite{TISHCHENKO2001}), or on some other variants of the diameter degree problem in more general sparse classes.

\paragraph*{Towards an exact bound}
We believe that in fact, our approach could allow to find the exact additive constant $c$ such that $\np_{\Delta,3}=\lfloor\tfrac92\Delta\rfloor + c$ for large values of $\Delta$: our proof of \Cref{thm: main1} implies that for every planar graph $G$ with diameter $3$ that has $\tfrac92\Delta + \Omega(1)$ vertices, we can construct a planar auxilliary graph $\Gamma$ with a neighbouring set $R$ of edges, such that $\mu^{\ast}(\Gamma[R])=\tfrac92$, such that $\Gamma$ is obtained from $G$ by emptying some nice lanterns, and adding some corresponding additional edges to $R$, each of them being either between two hubs, or pendant. In particular, by \Cref{cor: minimal_92} $\Gamma[R]$ contains a subgraph $H$ isomorphic to one of the three graphs $3K_3, K_3+3K_2, C_7+K_2$. Moreover, note that for every edge $e$ that belongs to a cycle of $H$, $\Gamma[R]$ cannot have an edge which is pendant and attaches to one of the two endvertices of $e$, otherwise $\Gamma[R]$ would either contain a matching of size $5$, or a copy of $H_1$, contradicting \Cref{lem: H1}. This implies that every edge of $R$ that belongs to a cycle of $H$ connects two hubs $u,v$ of a nice lantern $L$ of $G$, whose interior moreover only contains vertices that are adjacent to both $u$ and $v$. In particular, there are at most $\mu^{\ast}(\Gamma[R])$ such vertices. 

A way to obtain a more precise additive constant then consists in enumerating all possible ways of connecting the edges of $H$ to make it a neighbouring set of a planar graph, and, according to each possibility, show that every vertex of $G$ drawn in some face of the resulting planar graph can be converted into a weight $\tfrac{1}{\Delta}$ in some fractional matching of $\Gamma[R]$ (up to increasing a little bit $R$). However, we did not find a simple way to make this approach work within a reasonable number of cases, without making the size of this (already quite long) paper increase too much.

\paragraph*{Fractional domination number}
 A \emph{fractional dominating set} in a graph $G$ is a mapping $f: V(G)\to \mathbb R_{\geq 0}$ such that $\forall v\in V(G), \sum_{u\in N[v]}f(v)\geq 1$. We let $\gamma_f(G)$ denote the \emph{fractional dominating number} of $G$, that is, the infimum over the values of all fractional dominating sets of $G$. The dual of the fractional dominating number is the \emph{fractional packing number} $\rho_f(G)$, defined as the supremum over the values of all mappings $f: V(G)\to \mathbb R_{\geq 0}$ such that $\forall v\in V(G), \sum_{u\in N[v]}f(v)\leq 1$. By the strong duality theorem of linear programming, $\gamma_f(G)=\rho_f(G)$.
It is not hard to show that for every graph $G$,
 $$\frac{|V(G)|}{\Delta +1}\leq \gamma_f(G)\leq \gamma(G)$$
 
 Thus, for a graph $G$, finding upper-bounds on $\gamma_f(G)$, immediately gives upper bounds for $\tfrac{|V(G)|}{\Delta+1}$. On the other hand, we know that if $G$ is a sufficiently large planar graph of diameter 3, then $\gamma(G) \leq 6$~\cite{GH02}. A natural question is thus to determine an upper bound for $\gamma_f(G)$ when $G$ is a planar graph with diameter $3$. In particular, note that this upper bound is strictly greater than $\tfrac92$, as shown by the two graphs in \Cref{fig:extremal_fractional_dominating}, which have fractional dominating number $\tfrac{14}{3}>\tfrac92$. However, we could not find constructions of planar graphs $G$ with diameter $3$ for which $\gamma_f(G)>\tfrac{14}{3}$. 
 
\subsection*{Acknowledgement}
We would like to thank Ignasi Sau for pointing to us the reference \cite{alber2004polynomial} where the notion of region decomposition was introduced, which turned out to coincide with the collection of interiors we construct in the proof of \Cref{thm: main1} (see \Cref{rem: region-decomposition}). 

\begin{figure}[H]
    \centering
   \begin{subfigure}[b]{1\textwidth}
        \centering
        \includegraphics[scale=1.4]{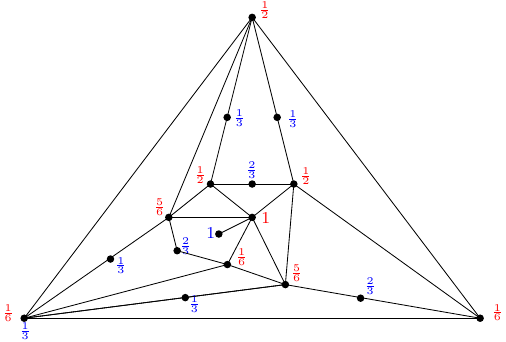}
        \caption{Graph with radius 2}
        \label{subfig:fractional_dominating_radius2}
    \end{subfigure}

    \begin{subfigure}[b]{1\textwidth}
        \centering
        \includegraphics[scale=1.4]{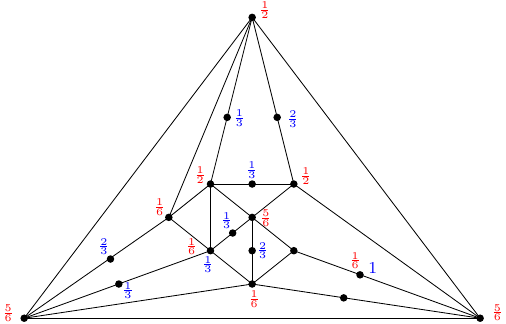}
        \caption{Graph which yields a construction on $\tfrac{9}{2}\Delta-3$ (by duplicating all vertices of degree 2)}
        \label{subfig:extremal_fractional_dominating}
    \end{subfigure}
    
    \caption{Examples of planar graphs with diameter 3 and $\gamma_f = \tfrac{14}{3}$. The red valuation gives a fractional dominating set showing that $\gamma_f\leq\tfrac{14}{3}$ and the blue valuation gives a fractional packing showing that $\rho_f\geq\tfrac{14}{3}$.}
    \label{fig:extremal_fractional_dominating}
\end{figure}


\bibliographystyle{alpha}
\bibliography{biblio}

\end{document}